\documentclass[12pt, reqno]{amsart}
\usepackage{amsmath}
\usepackage{amssymb}
\usepackage{epsfig}
\usepackage{graphicx}
\usepackage{color}
\usepackage{fullpage}
\usepackage{comment}
\definecolor{shadecolor}{gray}{0.875}
\usepackage{amscd}
\usepackage{stmaryrd}
\usepackage{amsmath}
\usepackage{ulem}
\usepackage{amsthm}
\usepackage{mathrsfs}
\usepackage{array}
\usepackage{booktabs}

\usepackage{csquotes}
\usepackage{enumitem}
\usepackage{graphics}
\usepackage{tikz-cd}

 \usepackage{array}
 \usepackage{booktabs}
 \usepackage{makecell}
 \usepackage{threeparttable}
\usepackage{longtable}

\usepackage{xcolor}
\definecolor{darkred}{RGB}{120,0,0}
\definecolor{darkgreen}{RGB}{0,90,45}

\usepackage[pagebackref]{hyperref}
\hypersetup{
	colorlinks=true,
	linkcolor=darkred,    % 定理、章节、公式等交叉引用
	citecolor=darkgreen,  % 文献引用
	urlcolor=darkred      % 网址
}

\usepackage{textcomp}

\let\oldtocsection=\tocsection

\let\oldtocsubsection=\tocsubsection

\let\oldtocsubsubsection=\tocsubsubsection

\renewcommand{\tocsection}[2]{\hspace{0em}\oldtocsection{#1}{#2}}
\renewcommand{\tocsubsection}[2]{\hspace{1em}\oldtocsubsection{#1}{#2}}
\renewcommand{\tocsubsubsection}[2]{\hspace{2em}\oldtocsubsubsection{#1}{#2}}
\makeatletter

\newcommand{\Rmnum}[1]{\expandafter\@slowromancap\romannumeral #1@}
\makeatother
\numberwithin{equation}{section}

\newcommand{\eg}{\textit{e}.\textit{g}.}

\input xy
\xyoption{all}

\calclayout
\allowdisplaybreaks[3]

\theoremstyle{plain}
\newtheorem{prop}{Proposition}[section]
\newtheorem{theo}[prop]{Theorem}
\newtheorem{coro}[prop]{Corollary}

\newtheorem{lemm}[prop]{Lemma}

\theoremstyle{definition}
\newtheorem{defi}[prop]{Definition}
\newtheorem{cons}[prop]{Construction}

\newtheorem{ques}[prop]{Question}

\newtheorem{rema}[prop]{Remark}

\newtheorem{exam}[prop]{Example}

\newtheorem{mainthm}{Theorem}

\newtheorem{maincor}[mainthm]{Corollary}

\def\ra{\rightarrow}

\def\cF{{\mathcal F}}

\def\cN{{\mathcal N}}
\def\cO{{\mathcal O}}

\def\bC{{\mathbb C}}

\def\bP{{\mathbb P}}
\def\bQ{{\mathbb Q}}
\def\bR{{\mathbb R}}
\def\bZ{{\mathbb Z}}

\def\Bl{\mathrm{Bl}}

\def\Eff{\overline{\mathrm{Eff}}}
\def\Pic{\mathrm{Pic}}

\def\Mor{\mathrm{Mor}}
\def\Nef{\mathrm{Nef}}

\def\Pic{\mathrm{Pic}}

\def\Proj{\mathrm{Proj}}

\makeatother
\makeatletter

\author{Runxuan Gao}
\address{Graduate School of Mathematics, Nagoya University, Furocho Chikusa-ku, Nagoya, 464-8602, Japan}
\email{m20015x@math.nagoya-u.ac.jp}

\title{The geometric exceptional set for Fano threefolds}

\begin{document}
\date{\today}
\maketitle

\begin{abstract}
	We study the geometric exceptional set in Manin's conjecture for smooth Fano threefolds. We prove that every adjoint rigid surface with $a$-invariant one on a non-product smooth Fano threefold is normal. This yields a description of the geometric exceptional set in terms of the higher-$a$ locus, anticanonical conics, and del Pezzo fibrations. Zariski density arises only from the latter two contributions.
	
	For del Pezzo fibrations, we construct a finite collection of parameter covers which control the associated breaking maps. We classify their low-genus components for general smooth Fano threefolds, yielding examples with potentially Zariski dense geometric exceptional sets in $22$ of the $105$ deformation families.
	
	Extending the construction of Batyrev and Tschinkel, we give examples of Fano fourfolds with quartic del Pezzo fibrations and an explicit Fano threefold with two del Pezzo fibrations of degree six. All these varieties have Zariski dense geometric exceptional sets. For the threefold, the original form of Manin’s conjecture fails on every nonempty open subset for every anticanonical height.
\end{abstract}

\tableofcontents

\section*{Introduction}

Manin's conjecture predicts the distribution of rational points of bounded
height on a Fano variety $X$ over a number field $F$ in terms of two birational
invariants, the $a$- and $b$-invariants (see Definitions~\ref{def-a-invariant} and~\ref{def-b-invariant}).
Rational points on certain subvarieties can contribute more points than the conjecture predicts for the ambient variety. It is therefore essential to remove an
exceptional subset of $X(F)$.

The original version of Manin's conjecture \cite{FMT89} predicts that
the exceptional set can be taken to be the set of $F$-rational points
of a proper closed subset of $X$.
The first examples in which the exceptional set must be Zariski dense
were found by Batyrev and Tschinkel \cite{BT96}.
These examples are Fano varieties of dimension at least three
(see \cite{gao2024geometric} for the description of a conjectural
exceptional set for them).
It was subsequently suggested by Peyre \cite{Peyre03} that the exceptional
set should be taken to be a thin set (see Definition~\ref{def-thin}).
Building on a series of studies of the two birational invariants
\cite{HTT15,LTT,LTDuke}, Lehmann, Sengupta, and Tanimoto proposed a
conjectural description of the exceptional set and proved that it is
thin \cite{LST}.
We refer to this set as the geometric exceptional set.
It is defined using breaking thin maps, which satisfy geometric
conditions indicating possible obstructions to the predicted
asymptotic; see Definition~\ref{defi-ges}.
Although this description is conjectural, known asymptotics or
sufficiently strong lower bounds for rational points on the sources
of such maps can sometimes show that every exceptional set must
be Zariski dense; see Remark~\ref{rema:ges-cs-es}.

The first examples of surfaces with Zariski dense exceptional sets were
found in \cite{gao2023zariski}.
Subsequently, weak del Pezzo surfaces with Zariski dense exceptional sets
were studied systematically in \cite{gao2024quasi}.
Combined with the framework developed in \cite{LTDuke}, these results
extend to the study of the geometric exceptional set on geometrically
rational surfaces; we recall the surface results used in this paper
in Section~\ref{sec2}.

Our first result concerns adjoint rigid surfaces with
$a$-invariant one (see Definition~\ref{def-adjoint-rigid}).
Previous work shows that when the general adjoint Iitaka fiber
of a dominant \(a\)-cover is a surface, its image in \(X\)
is a fiber of a del Pezzo fibration. We classify all adjoint
rigid surfaces in \(X\) with \(a\)-invariant one,
including those not known a priori to move in a dominant family.

\begin{mainthm}[=Theorem~\ref{ar:main}]\label{intro-ar-main}
	Let $X$ be a smooth Fano threefold over an algebraically closed field
	of characteristic zero, and assume that $X$ is not a product.
	Let $S\subset X$ be an integral surface such that
	$a(S,-K_X)=1$ and $(S,-K_X|_S)$ is adjoint rigid.
	Then $S$ is a normal del Pezzo surface with at most Du Val singularities,
	and $|S|$ is a base-point-free pencil defining a del Pezzo fibration
	\(
	X\longrightarrow\bP^1
	\)
	with $S$ as a reduced fiber.
\end{mainthm}

A normal adjoint rigid surface with $a$-invariant one can be isolated
on a terminal Fano threefold or on a smooth Fano fourfold
(Examples~\ref{ar:isolated-terminal-surface} and
\ref{ar:isolated-fourfold-surface}).
The theorem is proved as Theorem~\ref{ar:main}.

We now return to a number field $F$.
For a smooth Fano threefold $X$ such that $X_{\overline F}$
is not a product, Proposition~\ref{prop-breaking-map-reduction}
gives a precise description of $Z(X)$.
Besides the higher-$a$ contribution, it includes conic families,
del Pezzo fibers and their adjoint rigid covers, and dominant
del Pezzo base changes.
By Remark~\ref{rema-ges-density-reduction}, Zariski density
can arise only from the conic contribution or from smooth
breaking del Pezzo fibers and dominant breaking base changes.

For a del Pezzo fibration $\pi:X\to B$, we construct a finite
collection $P_\pi$ of dominant thin maps
\[
\tau_N:Y_N\longrightarrow X
\]
by base change along covers $N\to B$ parametrizing contractions
of the smooth fibers; see Construction~\ref{cons-dP-fib}.
These covers record invariant collections of disjoint
$(-1)$-curves and conic-bundle classes.
The following theorem describes how they control the associated
breaking maps.

\begin{mainthm}[Breaking maps from del Pezzo fibrations]
	\label{intro-dp-breaking}
	Let $X$ be a smooth Fano threefold over a number field $F$,
	and let $\pi:X\to B$ be a del Pezzo fibration.
	Write \(S\) for the generic fiber of \(\pi\).
	Then the following hold.
	\begin{enumerate}
		\item For every $\tau_N\in P_\pi$,
		\[
		a(Y_N,\tau_N^*(-K_X))=1,
		\qquad
		b(F,Y_N,\tau_N^*(-K_X))
		=\rho_{F(N)}(S_{F(N)}).
		\]
		In particular, $\tau_N$ is breaking if and only if
		\[
		\rho_{F(N)}(S_{F(N)})\geq\rho_F(X).
		\]
		
		\item Let $f:Y\to X$ be a dominant $a$-cover whose general
		geometric adjoint Iitaka fiber maps onto a general fiber of
		$\pi$.
		If $f$ is breaking, then it factors rationally over $X$
		through a member of $P_\pi$.
		If $\rho_F(X)=2$, the converse also holds.
	\end{enumerate}
\end{mainthm}

The theorem follows from Lemma~\ref{lemm-dp-base-change-b},
Theorem~\ref{theo-breaking},
Remark~\ref{rema-dp-candidates}, and
Corollary~\ref{cor-dp-rank-two}.
The essential observation is that a breaking base change increases
the arithmetic Picard rank of the generic fiber, so an additional
contraction becomes defined over the extended function field.
For quadric surface fibrations, the construction reduces to the
double cover parametrizing rulings; see Corollary~\ref{coro:p1xp1}.

For each \(\tau_N\in P_\pi\), we call the corresponding
finite morphism \(N\to B\) a parameter cover of \(\pi\).
We next classify the parameter covers of genus zero and one.
We say that the geometric exceptional set is potentially
Zariski dense if it becomes Zariski dense after a finite
extension of the ground field.

\begin{mainthm}[Low-genus parameter covers and potential density]
	\label{intro-dp-families}
	Among the deformation families of smooth Fano
	threefolds, there are exactly $22$ whose general member
	admits a del Pezzo fibration with a connected parameter
	cover of genus $0$ or $1$.
	In each of these families, the points
	represented by models over number fields with potentially
	Zariski dense geometric exceptional sets are Zariski
	dense in the corresponding geometric deformation space.
\end{mainthm}

A detailed statement is given in Theorem~\ref{thm:bt-families}.
The del Pezzo fibrations occurring in this classification have
degrees ranging from $4$ to $8$.
For a general member of each deformation family of smooth complex Fano
threefolds, Table~\ref{tab:bt-models} lists all del Pezzo fibrations with nonsplit generic
fiber. Proposition~\ref{prop:bt-covers} and Table~\ref{tab:bt-covers} describe the nontrivial components of their parameter covers.

As an application, we give explicit examples as follows.

\begin{mainthm}[Explicit density constructions]
	\label{intro-dp-density}
	Let $F$ be a number field.
	\begin{enumerate}
		\item There exists a family of smooth Fano fourfolds $X$ over $F$
		with a quartic del Pezzo fibration
		\(
		\pi:X\longrightarrow\bP^2_F
		\)
		such that $Z(X)$ is Zariski dense.
		
		\item General line sections
		of $\pi$
		are smooth Fano threefolds of type No.~2--7 with
		potentially Zariski dense geometric exceptional sets.
		
		\item Let $X$ be the Fano threefold of type No.~3--10 given by
		the blow up of $$\{x_0x_1+y_0y_1-z^2=0\}\subset\bP^4$$ along
		$x_0=x_1=0$ and $y_0=y_1=0$.
		Then for every
		non-empty open subset $U\subset X$ and every
		anticanonical height $H$, we have
		\[
		N(U(F),H,B)\gg_{U,H}B(\log B)^{4-1}.
		\]
	\end{enumerate}
\end{mainthm}

For part {\rm(1)}, see
Theorem~\ref{theo-dp4} and
Example~\ref{exam-diagonal-dp4}.
Part {\rm(2)} is proved in
Corollary~\ref{rema-dp-threefold-sections}.
Part {\rm(3)} follows from
Proposition~\ref{prop-symmetric-3-10-exceptional} and
Corollary~\ref{coro-symmetric-3-10-manin}.

The threefolds in part~{\rm(2)} form a special subfamily of No.~2--7,
its absence from the lists in
Theorem~\ref{intro-dp-families}.

An explicit member of the family in part {\rm(1)} is
given in Example~\ref{exam-diagonal-dp4}.
\begin{equation}\label{eq-intro-dp4}
	X=
	\left\{
	(x,y)\in\bP^2_F\times\bP^4_F
	\;\middle|\;
	\begin{aligned}
		y_0^2+y_1^2+y_2^2+y_3^2-y_4^2&=0,\\
		x_0(y_0^2-y_3^2)+x_1(y_1^2-y_3^2)
		+x_2(y_2^2-y_3^2)&=0
	\end{aligned}
	\right\}.
\end{equation}

The first projection defines a quartic del Pezzo fibration,
and $Z(X)$ is Zariski dense;
see Example~\ref{exam-diagonal-dp4}.

The quartic construction uses the two rulings of a general quadric in the net,
detected by a discriminant double cover.
For \eqref{eq-intro-dp4}, this cover is the degree-two Du Val del Pezzo surface
\[
w^2=-x_0x_1x_2(x_0+x_1+x_2)
\]
and is $F$-rational.
The example of Part {\rm(3)} admits two del Pezzo fibrations of degree $6$, each with a rational double cover parametrizing
split fibers of Picard rank 4.

We next apply the structural results to determine when the
geometric exceptional set cannot be Zariski dense.
For anticanonical conics, we use the geometry of their parameter
spaces together with the Mordell--Lang theorem.
Related arithmetic results for parameter spaces of lines and
conics were studied in
\cite[Section~4.2]{browning2017varieties} and
\cite[Sections~2.9--2.10]{corvaja2023integral}.
The del Pezzo contribution is constrained by Theorem~\ref{intro-dp-breaking}
and the numerical bound in Lemma~\ref{lemm-dp-complexity}.
Combining these results gives the following consequences for the
deformation families of smooth Fano threefolds.

\begin{maincor}[Non-density in the split case]\label{theo-1}
	Let $X$ be a smooth Fano threefold over a number field $F$ such that
	$X_{\overline F}$ is not a product and
	\[
	\rho_F(X)=\rho(X_{\overline{F}}).
	\]
	Then the geometric exceptional set of $X$ is not Zariski dense,
	except possibly in the following cases:
	\begin{enumerate}
		\item $X$ admits a del Pezzo fibration of degree $d\leq8$
		such that $\rho_F(X)\leq10-d$;
		\item $X$ has geometric Picard rank $1$, index $1$, and genus
		\[
		g\in\{2,3,4,5,6,10,12\};
		\]
		\item $X$ has geometric Picard rank $1$, index $2$, and degree
		\[
		d\in\{1,2,4,5\}.
		\]
	\end{enumerate}
	If $X_{\overline F}$ is general in its deformation family, the lists
	in the second and third cases reduce to
	$g\in\{2,10,12\}$ and $d\in\{4,5\}$, respectively.
\end{maincor}

The proof is given in Section~\ref{sec-proof}.
The del Pezzo case follows from Lemma~\ref{lemm-dp-complexity},
and the curve cases follow from
Theorem~\ref{theo-index-one-conics} and
Corollary~\ref{coro-potential-density}.
For index one, we give an explicit model of the
Mukai--Umemura threefold over $\mathbb Q$ whose geometric
exceptional set is Zariski dense
(Example~\ref{exam-mukai-umemura}).

In Section~\ref{sec-nonsplit},
we give the following results on non-split Fano threefolds.

\begin{maincor}[Non-split Fano threefolds of Picard rank one]
	\label{theo-1-nonsplit}
	Let $X$ be a smooth Fano threefold over a number field $F$ such that
	\[
	\rho_F(X)=1<\rho_{\mathrm{geom}}(X).
	\]
	Then the geometric exceptional set of $X$ is not Zariski dense,
	except possibly when $X_{\overline F}$ is of type
	\[
	\text{\rm No.~2-6\quad or\quad No.~2-21}.
	\]
	If, in addition, $X_{\overline F}$ is general in its deformation
	family, then the only remaining possible type is
	\[
	\text{\rm No.~2-21}.
	\]
\end{maincor}

The proof is given in Section~\ref{sec-proof}, using
Corollary~\ref{coro-nonsplit-conics}.

The paper is organized as follows.
Section~\ref{sec-gen} reviews the invariants and the definition
of the geometric exceptional set, and Section~\ref{sec2}
recalls the surface results used in the paper.
Section~\ref{sec-3} describes the higher-$a$ locus and proves
its purity.
Section~\ref{ar:section} proves Theorem~\ref{intro-ar-main}.
Section~\ref{sec-bsuface} describes the geometric exceptional
set and proves Theorem~\ref{intro-dp-breaking}.
Section~\ref{sec-dp-families} classifies the low-genus
parameter covers and proves Theorem~\ref{intro-dp-families}.
Section~\ref{sec-example} gives the explicit constructions
of Theorem~\ref{intro-dp-density}.
Sections~\ref{sec-curves} and~\ref{sec-nonsplit} study
anticanonical conics, including the non-split case.
Finally, Section~\ref{sec-proof} proves
Corollaries~\ref{theo-1} and~\ref{theo-1-nonsplit}.

\noindent
{\bf Acknowledgments.}
The author thanks Sho Tanimoto, Tim Browning and Brian Lehmann for useful conversations.
He also thanks Daniel Loughran and Ulrich Derenthal for their interest in this work.
The author was partially supported by Grant-in-Aid for JSPS Fellows Number 24KJ1234.

\noindent\textbf{AI Declaration}
The author began writing this paper in 2024 and developed
its main mathematical content before using AI tools.
Under the author's guidance, GPT-6 Astra carried out
some computations and screened cases in Section~\ref{sec-dp-families}.
The results were manually checked by the author.
Astra also brought several references to the author's attention,
including \cite{przyjalkowski2020bounds}, \cite{iliev2017hyper},
and \cite{jacovskis2023categorical}, of which the author had
previously been unaware. These helped the author
generalize some results from the previous manuscript in Section~\ref{sec-curves}.
Astra was also used to proofread the manuscript and provided
suggestions on language and organization.
The paper itself was written by the author
who is responsible for its correctness.

\section{Generalities}\label{sec-gen}

Throughout the paper, we work over a field $F$ of characteristic zero,
unless otherwise stated. Varieties are geometrically integral, and
all morphisms are defined over the ground field unless specified otherwise.
We write $\overline F$ for an algebraic closure of $F$.
For a projective variety $X$, let $N^1(X)$ be the real vector space of
Cartier divisors modulo numerical equivalence, and let
\[
\rho_F(X)=\dim N^1(X),
\qquad
\rho_{\mathrm{geom}}(X)=\rho_{\overline F}(X_{\overline F}).
\]
We write $\rho(X)=\rho_F(X)$ when the ground field is understood.
The pseudo-effective cone of divisors is denoted by $\Eff^1(X)$,
and its dual cone in $N_1(X)$ by $\Nef_1(X)$.

\subsection{Geometric invariants in Manin's conjecture}

\begin{defi}\label{def-a-invariant}
	Let $X$ be a smooth projective variety and let $L$ be a big
	$\bQ$-Cartier divisor on $X$. The $a$-invariant of $(X,L)$ is
	\[
	a(X,L):=\inf\{t\in\bR\mid K_X+tL\in\Eff^1(X)\}.
	\]
	For an arbitrary projective variety $X$, choose a resolution of
	singularities $\mu:\widetilde X\to X$ and define
	\[
	a(X,L):=a(\widetilde X,\mu^*L).
	\]
	When $L$ is nef but not big, we formally set $a(X,L)=+\infty$.
\end{defi}

This definition is independent of the choice of resolution by
\cite[Proposition~2.7]{HTT15}.
For a big divisor $L$, the inequality $a(X,L)>0$ holds if and only if
$X$ is geometrically uniruled \cite{BDPP12}.
If $L$ is big and nef and $X$ is geometrically uniruled, then
$a(X,L)$ is rational; see \cite[Theorem~2.16]{HTT15}.

\begin{defi}\label{def-b-invariant}
	Let $X$ be a smooth geometrically uniruled projective variety, and
	let $L$ be a big and nef $\bQ$-Cartier divisor on $X$. Set
	\[
	\cF(X,L):=
	\{\alpha\in\Nef_1(X)\mid (K_X+a(X,L)L)\cdot\alpha=0\}.
	\]
	The $b$-invariant of $(X,L)$ over $F$ is
	\[
	b(F,X,L):=\dim\cF(X,L).
	\]
	For an arbitrary geometrically uniruled projective variety $X$,
	we define $b(F,X,L):=b(F,\widetilde X,\mu^*L)$ on a resolution
	$\mu:\widetilde X\to X$.
	When $L$ is nef but not big, we formally set $b(F,X,L)=+\infty$.
\end{defi}

The definition is independent of the resolution by
\cite[Proposition~2.10]{HTT15}.
The face $\cF(X,L)$ is dual to the minimal face of $\Eff^1(X)$
containing $K_X+a(X,L)L$; see \cite[Lemma~4.14]{LST}.
We write $b(X,L)$ when $F$ is understood.

The $a$-invariant is unchanged under
field extension \cite[Corollary~4.5]{LST}, whereas the $b$-invariant
may increase under a finite extension \cite[Proposition~4.13]{LST}.
For a smooth weak Fano variety,
\begin{equation}\label{eq-fano-invariants}
	a(X,-K_X)=1,
	\qquad b(F,X,-K_X)=\rho_F(X).
\end{equation}

\subsection{Adjoint rigid pairs}\label{sec-adjoint-rigid-pairs}

In the birational study of Manin's conjecture, the adjoint Iitaka
fibration reduces the study of a polarized variety to that of
lower-dimensional adjoint rigid pairs.

\begin{defi}\label{def-adjoint-rigid}
	Let $X$ be a geometrically uniruled projective variety and let $L$
	be a big and nef $\bQ$-Cartier divisor on $X$.
	The pair $(X,L)$ is adjoint rigid if
	\[
	\kappa\bigl(\widetilde X,
	K_{\widetilde X}+a(X,L)\mu^*L\bigr)=0,
	\]
	where $\mu:\widetilde X\to X$ is a resolution of singularities.
\end{defi}

The Iitaka dimension in Definition~\ref{def-adjoint-rigid} is
independent of the resolution.
Put $D=K_{\widetilde X}+a(X,L)\mu^*L$, and choose a positive integer
$m$ such that $mD$ is Cartier.
The graded algebra
\[
R(\widetilde X,mD)
=\bigoplus_{n\geq0}
H^0\bigl(\widetilde X,\cO_{\widetilde X}(nmD)\bigr)
\]
is finitely generated by \cite[Corollary~1.1.2]{BCHM10}.
The induced rational map $X\dashrightarrow\Proj R(\widetilde X,mD)$ is called
the adjoint Iitaka fibration associated to $(X,L)$.
%After resolving this map, we take its model with connected fibers.
After resolving this map, a general geometric fiber is an adjoint rigid pair with
the same $a$-invariant as $(X,L)$ by \cite[Lemma~4.3]{LST}.

\subsection{Thin sets}

We recall the notion of a thin set used in the formulation of
Manin's conjecture; see \cite[Definition~2.1]{LST}.

\begin{defi}\label{def-thin}
	Let $X$ and $Y$ be varieties over $F$.
	A morphism $f:Y\to X$ is a thin map if it is generically finite
	onto its image and satisfies one of the following conditions:
	\begin{enumerate}[label={\rm(\arabic*)}]
		\item $f(Y)$ is not dense in $X$; in this case $f$ is of type I;
		\item $f$ is dominant and has degree at least $2$; in this case
		$f$ is of type II.
	\end{enumerate}
	A subset $Z\subset X(F)$ is thin if it is contained in a finite union
	$\bigcup_i f_i(Y_i(F))$ for thin maps $f_i:Y_i\to X$ over $F$.
\end{defi}

In particular, the inclusion of a proper subvariety and the map
from a resolution of that subvariety to $X$ are both of type I.

\subsection{Geometric exceptional sets}

We first recall two notions used in the definition of the geometric
exceptional set.

\begin{defi}\label{def-a-cover}
	Let $X$ and $Y$ be smooth projective varieties, and let $L$ be a big
	and nef $\bQ$-Cartier divisor on $X$ with $a(X,L)>0$.
	A dominant generically finite morphism $f:Y\to X$ of degree at
	least $2$ is an $a$-cover of $(X,L)$ if
	\[
	a(Y,f^*L)=a(X,L).
	\]
	It is an adjoint rigid $a$-cover if, in addition, $(Y,f^*L)$ is
	adjoint rigid in the sense of Definition~\ref{def-adjoint-rigid}.
\end{defi}

For an $a$-cover, pushforward of curve classes induces a map
\[
f_*:\cF(Y,f^*L)\longrightarrow\cF(X,L);
\]
see \cite[Lemma~4.25]{LST}.

\begin{defi}\label{def-face-contracting}
	An $a$-cover $f:Y\to X$ is face contracting if the induced map
	$f_*:\cF(Y,f^*L)\to\cF(X,L)$ is not injective.
\end{defi}

We now assume that $F$ is a number field and recall the conjectural
exceptional set proposed in \cite[Definition~5.1]{LST}, specialized
to anticanonically polarized Fano varieties.

\begin{defi}[Geometric exceptional set]\label{defi-ges}
	Let $X$ be a smooth Fano variety over $F$ and put $L=-K_X$.
	The geometric exceptional set $Z(X)$ is the union of $f(Y(F))$,
	where $Y$ is smooth, projective, and geometrically uniruled, and
	$f:Y\to X$ is a thin map satisfying one of the following conditions:
	\begin{enumerate}[label={\rm(\arabic*)}]
		\item\label{ges-a>1}
		\(
		\bigl(a(Y,f^*L),b(F,Y,f^*L)\bigr)
		>
		\bigl(a(X,L),b(F,X,L)\bigr);
		\)
		\item the two pairs are equal, and one of the following holds:
		\begin{enumerate}[label={\rm(\alph*)}]
			\item $\dim Y<\dim X$;
			\item $\dim Y=\dim X$ and $\kappa(Y,K_Y+f^*L)>0$;
			\item $\dim Y=\dim X$, $\kappa(Y,K_Y+f^*L)=0$, and $f$ is
			face contracting.
		\end{enumerate}
	\end{enumerate}
	The pairs are compared in lexicographic order.
	A thin map satisfying these conditions is called a breaking thin map.
\end{defi}

\begin{rema}\label{rema:ges-cs-es}
	We distinguish the geometric exceptional set $Z(X)$ from an
	exceptional set for which Manin's predicted asymptotic holds.
	The former is a conjectural candidate for the latter, motivated
	by the geometric side of the predicted asymptotic.
	We refer to $Z(X)$ as an exceptional set only when the predicted
	asymptotic holds on $X(F)\setminus Z(X)$.
	In specific cases, known asymptotics or sufficiently strong
	lower bounds for rational points on the sources $Y$ of breaking
	thin maps can show that every exceptional set for $X$ must be
	Zariski dense
	(see \eg~Corollary~\ref{coro-symmetric-3-10-manin}). 
\end{rema}

The set $Z(X)$ is thin by \cite[Theorem~5.7]{LST}.
We say that it is potentially Zariski dense if there is a finite
extension $K/F$ such that the geometric exceptional set of $X_K$
is Zariski dense in $X_K$.

\section{Geometric exceptional sets on surfaces}\label{sec2}

We recall the surface results used in the study of exceptional sets
on Fano threefolds. A weak del Pezzo surface is a smooth projective
surface $S$ such that $-K_S$ is big and nef; it is a del Pezzo surface
if $-K_S$ is ample. Its degree is $d=K_S^2$.
By \eqref{eq-fano-invariants},
\[
a(S,-K_S)=1,
\qquad b(F,S,-K_S)=\rho_F(S).
\]
%For weak del Pezzo surfaces, the exceptional set is understood in
%the sense of \cite[Definition~5.1]{LST}; in particular, it includes
%the rational points on the $(-2)$-curves, where $-K_S$ has degree zero.

We first work over an algebraically closed field to describe curves
with $a$-invariant at least one.
We recall the following classification,
following \cite[Lemma~9.2 and its proof]{LT19}.

\begin{prop}\label{prop-dim2}
	Let $S$ be a weak del Pezzo surface of degree $d$ over an
	algebraically closed field of characteristic zero, and let
	$C\subset S$ be an integral curve.
	If $a(C,-K_S|_C)>1$, then $C$ is one of the following:
	\begin{enumerate}[label={\rm(\arabic*)}]
		\item a $(-2)$-curve, in which case $a(C,-K_S|_C)=+\infty$;
		\item a $(-1)$-curve;
		\item a rational curve in $|-K_S|$, when $d=1$.
	\end{enumerate}
	In the last two cases, $a(C,-K_S|_C)=2$.
	
	If $a(C,-K_S|_C)=1$, then $C$ is one of the following:
	\begin{enumerate}[label={\rm(\arabic*)},resume]
		\item a smooth fiber of a conic bundle $S\to\bP^1$;
		\item a rational curve in $|-K_S|$, when $d=2$;
		\item\label{case-dp1-exceptional}
		a rational curve in $|-K_S+E|$, when $d=1$, where $E$ is
		an effective divisor satisfying $E^2=-1$ and $(-K_S)\cdot E=1$;
		\item a rational curve in $|-2K_S|$, when $d=1$.
	\end{enumerate}
\end{prop}

We now return to a number field $F$.
For a smooth del Pezzo surface, \cite[Theorem~9.1]{LT19} shows that
the geometric exceptional set is not Zariski dense, except possibly
when $d=1$ and $\rho_F(S)=1$; this case is also excluded when
$S_{\overline F}$ is general in moduli.
A degree-one del Pezzo surface of Picard rank one with Zariski dense
geometric exceptional set was constructed in \cite{gao2023zariski}.

As for dominant breaking maps, \cite[Theorem~6.2]{LTDuke} shows that a smooth
del Pezzo surface admits no adjoint rigid $a$-cover in the sense of
Definition~\ref{def-a-cover}.
A weak del Pezzo surface may admit such covers. Up to birational
equivalence over an algebraic closure, they correspond to finite
covers of its anticanonical model that are \'etale in codimension one
\cite[Corollary~6.9]{LTDuke}.
Such maps are called quasi-\'etale covers. The anticanonical model
is a del Pezzo surface with at most Du Val singularities, and the
quasi-\'etale covers of these surfaces, together with their arithmetic
contributions to the exceptional set, are studied in
\cite{gao2024quasi}.
If $T$ is a Du Val del Pezzo surface and $r:S\to T$ is its minimal
resolution, then $K_S=r^*K_T$ and Definition~\ref{def-b-invariant} gives
\[
b(F,T,-K_T)=\rho_F(S).
\]

Finally, for a smooth geometrically uniruled surface with a big
and nef polarization $L$, the adjoint Iitaka dimension is either
zero or one. In dimension one, the geometry is governed by the
rational fibers of the canonical map. In the adjoint rigid case,
the adjoint minimal model program reduces the study to the
anticanonical polarization on a weak del Pezzo surface; see
\cite[Proposition~5.9]{LTT} and \cite[Theorems~6.1 and~6.10]{LTDuke}.
%This explains the role of weak del Pezzo surfaces in the study of
%exceptional sets on geometrically rational surfaces.

\section{The \(a\)-exceptional locus}\label{sec-3}

We work over an algebraically closed field $F$
of characteristic zero.
Let $X$ be a smooth Fano threefold.
For each positive-dimensional integral subvariety $Y\subsetneq X$,
choose a projective resolution $\mu_Y:\widetilde Y\to Y$, and denote
its composition with $Y\hookrightarrow X$ by the same symbol.
Since $a(X,-K_X)=1$, we define its $a$-exceptional set by
\begin{equation}\label{eq-a-exceptional-locus}
	Z_a(X):=
	\bigcup_{\substack{Y\subsetneq X\text{ integral},\ \dim Y>0\\
			a(Y,-K_X|_Y)>1}}\mu_Y\bigl(\widetilde Y(F)\bigr),
\end{equation}
where the invariant is computed on a resolution as in
Definition~\ref{def-a-invariant}.
Since $F$ is algebraically closed, this set consists of all closed
points of the union of the images $\mu_Y(\widetilde Y)$.
By \cite[Theorem~3.2]{beheshti2022moduli}, this union is a proper closed
subset of $X$, and each of its irreducible components has
$a$-invariant greater than one.
In this section, we show that this geometric image is either
empty or pure of dimension two in Proposition~\ref{prop-puredimension2}, and
describe its components in Corollary~\ref{cor-a-exceptional-set}.

We first recall the classification of surfaces with higher
$a$-invariant. We use Mori's notation E1--E5 for divisorial
extremal contractions; see \cite[Theorem~2.2]{beheshti2022moduli}.
An integral rational curve $C\subset X$ with $-K_X\cdot C=1$
will be called an anticanonical line. Such a curve can occur only
when $X$ has index one. If $-K_X$ is very ample, these are precisely
the lines in the anticanonical embedding.

\begin{prop}[{\cite[Theorem~4.1]{beheshti2022moduli}}]
	\label{prop-a-higher}
	Let $Y\subset X$ be an integral surface with
	$a_Y:=a(Y,-K_X|_Y)>1$.
	Let $\psi:\widetilde Y\to Y$ be a resolution and set
	$D_Y=K_{\widetilde Y}+a_Y\psi^*(-K_X|_Y)$.
	Then one of the following holds:
	\begin{enumerate}[label={\rm(\arabic*)},ref=\arabic*]
		\item\label{case-iitaka0}
		$\kappa(\widetilde Y,D_Y)=0$.
		In this case, $Y$ is one of the following:
		\begin{enumerate}[label={\rm(\alph*)}]
			\item An E5 divisor, with
			$(Y,-K_X|_Y)\cong(\bP^2,\cO(1))$ and $a_Y=3$.
			\item An E1 divisor, with
			$(Y,-K_X|_Y)\cong(\bP^1\times\bP^1,\cO(1,1))$ and $a_Y=2$.
			Both rulings correspond to E1 contractions.
			\item An E3 or E4 divisor, with
			$(Y,-K_X|_Y)\cong(Q,\cO_Q(1))$ and $a_Y=2$, where
			$Q\subset\bP^3$ is a smooth quadric in the E3 case and a quadric
			cone in the E4 case.
			\item An E2 divisor, with
			$(Y,-K_X|_Y)\cong(\bP^2,\cO(2))$ and $a_Y=\frac32$.
		\end{enumerate}
		\item\label{case-iitaka1}
		$\kappa(\widetilde Y,D_Y)=1$.
		In this case, $Y$ is swept out by anticanonical lines.
	\end{enumerate}
\end{prop}

To describe the whole locus \eqref{eq-a-exceptional-locus}, it
remains to show that every curve with higher $a$-invariant is
contained in a surface with higher $a$-invariant.
In the setting of \cite[Corollary~2.1.5(ii)]{kuznetsov2018hilbert},
the required deformation bound follows from the Hilbert scheme
of lines. For arbitrary $X$, we work with morphisms from $\bP^1$;
this also treats singular rational curves and avoids the extra
subschemes discussed in \cite[Remark~2.1.3]{kuznetsov2018hilbert}.

\begin{prop}\label{prop-puredimension2}
	Every integral curve $C\subset X$ with $a(C,-K_X|_C)>1$ moves
	in a positive-dimensional family of anticanonical lines and is
	contained in an integral surface $Y\subset X$ with
	$a(Y,-K_X|_Y)>1$.
	In particular, the union of the images $\mu_Y(\widetilde Y)$ in
	\eqref{eq-a-exceptional-locus} is either empty or pure of dimension two.
\end{prop}

\begin{proof}
	Let $\nu:\widetilde C\to C$ be the normalization and let
	$g=g(\widetilde C)$. Definition~\ref{def-a-invariant} gives
	\[
	a(C,-K_X|_C)=\frac{2-2g}{-K_X\cdot C}.
	\]
	Since $-K_X$ is ample and Cartier, the denominator is a positive
	integer. Thus $a(C,-K_X|_C)>1$ forces $g=0$ and
	$-K_X\cdot C=1$.
	
	Let $f:\bP^1\to X$ be the normalization map followed by the
	inclusion, and let $M$ be an irreducible component of
	$\Mor(\bP^1,X)$ containing $[f]$, with its reduced structure.
	Deformation theory and Riemann--Roch give
	\[
	\dim M\geq
	h^0(\bP^1,f^*T_X)-h^1(\bP^1,f^*T_X)
	=-K_X\cdot C+3=4.
	\]
	Every map parametrized by $M$ has anticanonical degree one.
	It is therefore birational onto its image, which is an
	anticanonical line. Two such maps have the same image precisely
	when they differ by an automorphism of $\bP^1$.
	Since $\dim\operatorname{PGL}_2=3$, the images vary in a family
	of dimension at least one.
	
	Let $W$ be the closure of the image of the evaluation map
	$\operatorname{ev}:\bP^1\times M\to X$.
	Then $W$ is irreducible, contains $C$, and has dimension at least
	two. Every curve in the family has $a$-invariant two, so the
	evaluation image lies in the union of the images of the resolutions
	in \eqref{eq-a-exceptional-locus}. Since this union is closed and
	proper, it contains $W$ and $\dim W=2$.
	Thus $W$ is an irreducible component of this union, and
	\cite[Theorem~3.2]{beheshti2022moduli} gives
	$a(W,-K_X|_W)>1$.
	
	Every curve occurring in \eqref{eq-a-exceptional-locus} is
	therefore contained in such a surface. As the remaining
	subvarieties in that union are already surfaces, the nonempty
	union of the images $\mu_Y(\widetilde Y)$ is pure of dimension two.
\end{proof}

Let $\mathcal L(X)$ denote the closure of the union of all
anticanonical lines on $X$. The preceding results give the
following description of the $a$-exceptional locus.

\begin{coro}\label{cor-a-exceptional-set}
	We have
	\begin{equation}\label{eq-a-exceptional-description}
%		\bigcup_{\substack{Y\subsetneq X\text{ integral},\ \dim Y>0\\
%				a(Y,-K_X|_Y)>1}}\mu_Y(\widetilde Y)
		Z_a(X)=\mathcal L(X)\cup\bigcup_{E\text{ of type E2}}E,
	\end{equation}
	where an E2 divisor is defined by
	$(E,-K_X|_E)\cong(\bP^2,\cO(2))$.
\end{coro}

\begin{proof}
	Anticanonical lines have $a$-invariant $2$, and E2 divisors have
	$a$-invariant $\frac32$. Since the left-hand side of
	\eqref{eq-a-exceptional-description} is closed, it contains the
	right-hand side.
	
	Conversely, Proposition~\ref{prop-puredimension2} shows that every
	irreducible component $Y$ of this union is a surface with higher
	$a$-invariant. Apply Proposition~\ref{prop-a-higher}.
	In case~\textup{(\ref{case-iitaka1})}, $Y$ is swept out by
	anticanonical lines. In case~\textup{(\ref{case-iitaka0})}, either
	$Y$ is an E2 divisor, or the lines on the polarized surface listed
	there sweep out $Y$ and have anticanonical degree one in $X$.
	This proves the assertion.
\end{proof}

For a smooth Fano threefold $X$ over a number field $F$, we use
the same definition \eqref{eq-a-exceptional-locus}, with the union
taken over geometrically integral subvarieties $Y$ defined over $F$
and resolutions $\mu_Y$ defined over $F$.
The resulting set is independent of the chosen resolutions by
Lang--Nishimura over $X$.
Every map $\mu_Y$ in the union is breaking, so $Z_a(X)\subset Z(X)$.
The geometric union obtained by applying
\eqref{eq-a-exceptional-locus} to $X_{\overline F}$ is Galois
invariant by the invariance of $a$-invariant under field extension.
It therefore descends to a proper reduced closed subset of $X$
containing $Z_a(X)$.

\section{Adjoint rigid surfaces in smooth Fano threefolds}
\label{ar:section}

Throughout this section, the ground field is algebraically closed of
characteristic zero. Let $X$ be a smooth Fano threefold and put $L=-K_X$.

By \cite[Theorem~5.3]{beheshti2022moduli} and
\cite[Theorem~3.10]{burke2022geometric}, every dominant family of
adjoint rigid surfaces with $a$-invariant one on a non-product
smooth Fano threefold defines a del Pezzo fibration.
For Manin's conjecture, however, individual adjoint rigid surfaces
must be accounted for in the exceptional set, whether or not they
belong to a dominant family.
In this section, we show that every adjoint rigid surface with
$a$-invariant one on a non-product smooth Fano threefold is a
normal fiber of a del Pezzo fibration.

\begin{theo}\label{ar:main}
	Let $S\subset X$ be an integral surface such that
	\begin{equation}\label{ar:assumption}
		a(S,L)=1,\qquad (S,L)\text{ is adjoint rigid}.
	\end{equation}
	If $X$ is not a product, then $S$ is a normal del Pezzo surface with at most
	Du Val singularities, and $|S|$ is a base-point-free pencil defining a del
	Pezzo fibration $\pi:X\to\bP^1$, with $S$ as a reduced fiber.
\end{theo}

The smoothness and dimension assumptions in Theorem~\ref{ar:main}
are essential. A normal adjoint rigid surface with $a$-invariant one
can be non-movable on a terminal Fano threefold or on a smooth Fano
fourfold, as the following examples show.

\begin{exam}[a terminal Fano threefold]
	\label{ar:isolated-terminal-surface}
	Let $q:Y=(\bP^1)^3\to X=Y/G$ be the quotient by
	$G\cong\bZ/2\bZ$, acting in affine coordinates by
	$\sigma(x,y,z)=(-x,-y,-z)$. The singular locus of $X$ consists of
	eight terminal points of type $\frac12(1,1,1)$, so $X$ is not a
	product. Since $q$ is \'etale in codimension one,
	$q^*(-K_X)=-K_Y$ and $L=-K_X$ is ample.
	
	Put $Y_0=\bP^1\times\bP^1\times\{0\}$ and $S_0=q(Y_0)$.
	Then $S_0$ is a del Pezzo surface of degree four with four $A_1$
	singularities.
	The restriction $q_0:Y_0\to S_0$ is also \'etale in codimension one,
	so adjunction on $Y_0$ gives
	\[
	q_0^*K_{S_0}
	\sim_{\bQ} K_{Y_0}
	\sim (K_Y+Y_0)|_{Y_0}
	\sim K_Y|_{Y_0}
	\sim_{\bQ} q_0^*(K_X|_{S_0}).
	\]
	Since $q_0$ is finite and surjective, it follows that
	\[
	K_{S_0}\sim_{\bQ}K_X|_{S_0}.
	\]
	Thus $(S_0,L|_{S_0})$ has
	$a$-invariant $1$ and is adjoint rigid.
	The numerical class of $S_0$ contains only two integral surfaces,
	since the pullback of any such surface to $Y$ is linearly equivalent
	to $Y_0$ and hence is a $G$-invariant fiber of the third projection.
\end{exam}

\begin{exam}[cubic fourfolds containing a plane]
	\label{ar:isolated-fourfold-surface}
	Let $X\subset\bP^5$ be a smooth cubic fourfold containing a plane
	$S\simeq\bP^2$, and let $H$ be the hyperplane class.
	Since
	\[
	(-K_X)|_S\sim3H|_S\sim-K_S,
	\]
	the surface $S$ has $a(S,-K_X)=1$ and is adjoint rigid.
	Any connected flat family of integral surfaces through $S$ consists
	of planes, since the $H$-degree remains one.
	A smooth cubic fourfold contains only finitely many planes
	\cite[Lemma~2.4]{DegtyarevItenbergOttem2023}, so $S$ is non-movable.
\end{exam}

The proof of Theorem~\ref{ar:main} has two ingredients.
Lemmas~\ref{ar:reduction} and~\ref{ar:conductor} show that $S$ is nef
and satisfies $1\le L^2S\le9$.
We exclude the remaining types using three
geometric obstructions: an obstruction for pullbacks under conic
bundles (Lemma~\ref{ar:conic}), an obstruction arising from blow-ups
along curves of positive genus (Lemma~\ref{ar:blowdown}
and Proposition~\ref{ar:elliptic}), and an obstruction for double covers
(Lemma~\ref{ar:double} and Proposition~\ref{ar:double-exclusion}).

For completeness, we also describe the product cases as follows.
\begin{prop}[Surfaces on products]\label{prop-product-adjoint-surfaces}
	Over an algebraically closed field of characteristic zero, let
	$X=\bP^1\times S$, with $S$ a smooth del Pezzo surface.
	An integral surface $T\subset X$ with $a(T,-K_X|_T)=1$ is adjoint rigid
	if and only if it has one of the following forms:
	\begin{enumerate}
		\item $T=\{t\}\times S$, or
		\item $T=\bP^1\times C$,
		where $C\subset S$ is an integral rational curve with $(-K_S)\cdot C=2$.
	\end{enumerate}
\end{prop}

\begin{proof}
	Suppose that $T$ dominates $S$, and resolve $T$ to obtain
	$h:\widetilde T\to S$ and $p:\widetilde T\to\bP^1$.
	The ramification formula gives
	\[
	K_{\widetilde T}+(-K_X)|_{\widetilde T}
	\sim R_h+2p^*\cO_{\bP^1}(1),
	\]
	where $R_h$ is an effective divisor.
	If $p$ is nonconstant, this implies that
	$\kappa(K_{\widetilde T}+(-K_X)|_{\widetilde T})\geq1$,
	contradicting the assumption that $T$ is adjoint rigid.
	Thus $p$ is constant and $T=\{t\}\times S$.
	
	Otherwise, the image of $T$ in $S$ is an integral curve $C$, and
	$T=\bP^1\times C$.
	Let $\nu:\widetilde C\to C$ be the normalization and
	$q:\bP^1\times\widetilde C\to\widetilde C$ the second projection.
	The adjoint divisor on $\bP^1\times\widetilde C$ is
	\[
	q^*\bigl(K_{\widetilde C}+\nu^*((-K_S)|_C)\bigr).
	\]
	Since $q_*\cO_{\bP^1\times\widetilde C}=\cO_{\widetilde C}$,
	its Iitaka dimension equals that of
	$K_{\widetilde C}+\nu^*((-K_S)|_C)$.
	This Iitaka dimension is zero precisely when
	$\widetilde C\cong\bP^1$ and $(-K_S)\cdot C=2$.
\end{proof}

\subsection{Numerical restrictions}

Write $\nu:T\to S$ for the normalization, $g:T\to X$ for its composition
with the inclusion, and $r:\Sigma\to T$ for the minimal resolution. Put
$f=g\circ r$.

\begin{lemm}[{\cite[Lemma~5.3]{Lehmann2019}}]
	\label{ar:reduction}
	Under \eqref{ar:assumption}, $T$ is a del Pezzo surface with at most Du Val
	singularities, $\Sigma$ is a weak del Pezzo surface, and
	\begin{equation}\label{ar:crepant}
		\omega_T\simeq g^*\omega_X,\qquad
		\omega_\Sigma\simeq f^*\omega_X.
	\end{equation}
	Consequently,
	\begin{equation}\label{ar:degree}
		-K_\Sigma\sim f^*L,\qquad
		d:=L^2S=K_T^2=K_\Sigma^2\in\{1,\ldots,9\}.
	\end{equation}
\end{lemm}

The following lemma shows that a surface satisfying \eqref{ar:assumption}
must be nef.
\begin{lemm}\label{ar:conductor}
	Let $X$ be a smooth Fano threefold,
	and let $S$ be an integral surface on $X$ satisfying \ref{ar:assumption}.
	Then $S$ is a nef divisor on $X$, and $S$ is normal if and only if
	$LS^2=0$.
\end{lemm}
\begin{proof}
	Let $S$ be a non-nef prime divisor on $X$.
	By the cone theorem \cite[Theorem~3.7]{Kollr1998},
	there is an extremal ray $R$ of the Mori cone with $S\cdot R<0$.
	Let $\varphi:X\to Z$ be the associated contraction.
	Since $S$ is effective, every curve
	contracted by $\varphi$ is contained in $S$. Thus $\varphi$ is
	birational, and \cite[Theorem~1.32]{Kollr1998} implies that it is a
	divisorial contraction of type E1--E5.
	Thus $a(S,L)>1$ by Proposition~\ref{prop-a-higher}, and $S$ does not
	satisfy \ref{ar:assumption}.
	
	If $S$ is normal, then Lemma~\ref{ar:reduction} implies that
	$N_{S/X}\cong \cO_S$, which implies that $LS^2=0$.
	If $S$ is non-normal, let $\nu:T\to S$ be its normalization
	and $g:T\to X$ the natural map.
	By finite duality \cite[Proposition~2.3]{Reid94} and
	Lemma~\ref{ar:reduction}, the conductor defines a nonzero
	effective Cartier divisor $\Delta$ on $T$ satisfying
	\[
	\mathcal O_T(-\Delta)
	\simeq \omega_T\otimes\nu^*\omega_S^{-1}
	\simeq g^*\mathcal O_X(-S).
	\]
	Since $g^*L$ is ample, the projection formula gives
	\(
	LS^2=g^*L\cdot\Delta>0.
	\)
	\end{proof}

We record some numerical consequences that will be used in the proof.

\begin{lemm}\label{ar:numerics}
	Suppose that $S$ satisfies \eqref{ar:assumption}.
	\begin{enumerate}[label={\rm(\arabic*)},ref=\arabic*]
		\item For every $D\in\operatorname{Pic}(X)$,
		\begin{equation}\label{ar:parity}
			D^2S+LDS\in2\bZ.
		\end{equation}
		\item If $R\subset\Sigma$ is an integral curve and
		$0<(-K_\Sigma)R<d$, then $R$ is a smooth rational curve.
	\end{enumerate}
\end{lemm}

\begin{proof}
	Riemann--Roch on $\Sigma$ gives
	\[
	\chi(\Sigma,f^*D)=1+\frac{D^2S+LDS}{2},
	\]
	which proves (1). For (2), put $m=(-K_\Sigma)R$.
	By the Hodge index theorem,
	\[
	R^2\le \frac{(f^*L\cdot R)^2}{(f^*L)^2}
	=\frac{m^2}{d}<m.
	\]
	Adjunction therefore gives
	\[
	2p_a(R)-2=R\cdot(R+K_\Sigma)=R^2-m<0.
	\]
	It follows that
	$p_a(R)=0$.
\end{proof}

\subsection{Geometric obstructions}

We will use the following result to exclude Picard rank $1$ cases.

\begin{lemm}[A differential obstruction]\label{ar:cotangent}
	Let $X\subset P:=\bP(w_0,\ldots,w_N)$ be a smooth threefold that is
	a complete intersection of positive degrees $c_1,\ldots,c_{N-3}$ and is
	contained in the smooth locus of $P$. Suppose that
	$H=\mathcal O_X(1)$ is an ample line bundle. There is no morphism
	$f:W\to X$ from a smooth projective surface, generically finite onto its image,
	such that $\omega_W\simeq f^*\omega_X$.
\end{lemm}

\begin{proof}
	Since $f$ is generically finite onto a surface, the second exterior power of $df$ gives a nonzero section
	\[
	\bigwedge\nolimits^2df\in
	H^0\!\left(W,(f^*\Omega_X^2)^\vee\otimes\omega_W\right).
	\]
	Since $\Omega_X^1$ has rank three, we have
	\[
	(\Omega_X^2)^\vee\simeq\Omega_X^1\otimes\omega_X^{-1}.
	\]
	Thus, using $\omega_W\simeq f^*\omega_X$, we obtain
	\begin{equation}\label{eq:nonzero-section}
		0\ne\bigwedge\nolimits^2df\in H^0(W,f^*\Omega_X^1).
	\end{equation}
	
	On the other hand, since $M=f^*H$ is nef and big, for every integer $c>0$,
	we have
	\[
	H^0(W,\mathcal O_W(-cM))=H^1(W,\mathcal O_W(-cM))=0,
	\]
	by Serre duality and Kawamata--Viehweg vanishing.
	Since $X$ is contained in the smooth locus of $P$,
	we have the weighted Euler sequence (see \cite[\S\S2.1.4--2.1.5]{Dolgachev82})
	\[
	0\longrightarrow\Omega_P^1|_X
	\longrightarrow\bigoplus_{i=0}^N\mathcal O_X(-w_i)
	\longrightarrow\mathcal O_X\longrightarrow0
	\]
	and the conormal sequence
	\[
	0\longrightarrow\bigoplus_{j=1}^{N-3}\mathcal O_X(-c_j)
	\longrightarrow\Omega_P^1|_X
	\longrightarrow\Omega_X^1\longrightarrow0.
	\]
	Their pullbacks to $W$ remain exact, since each sequence is locally split
	as a sequence of modules. The two vanishings imply successively
	$H^0(W,f^*(\Omega_P^1|_X))=0$ and
	$H^0(W,f^*\Omega_X^1)=0$, a contradiction to \eqref{eq:nonzero-section}.
\end{proof}

The differential obstruction allows us to rule out the case of Picard
number one.
\begin{prop}[Picard number one]\label{ar:rank1-exclusion}
	A smooth Fano threefold of Picard number one contains no surface satisfying
	\eqref{ar:assumption}.
\end{prop}

\begin{proof}
	Write $\operatorname{Pic}(X)=\bZ H$, $L=iH$, $H^3=v$, and $S\sim mH$,
	where $m\ge1$. Lemma~\ref{ar:reduction} gives $i^2mv\le9$. The
	classification of smooth Fano threefolds of Picard number one therefore
	leaves only index one with $v=2,4,6,8$, or index two with $v=1,2$.
	The corresponding models are the sextic double solid, the quartic or
	double quadric, the complete intersections of types $(2,3)$ and
	$(2,2,2)$, and the del Pezzo threefolds $V_1,V_2$.
	They are respectively weighted complete intersections
	\[
	\begin{gathered}
		X_6\subset\bP(1^4,3),\quad
		X_4\subset\bP^4\ \text{or}\
		X_{2,4}\subset\bP(1^5,2),\quad
		X_{2,3}\subset\bP^5,\quad X_{2,2,2}\subset\bP^6,\\
		V_1=X_6\subset\bP(1^3,2,3),\qquad
		V_2=X_4\subset\bP(1^4,2).
	\end{gathered}
	\]
	The ambient weighted projective spaces are well formed and have
	at most isolated singularities. Thus these threefolds are well
	formed. Since they are smooth, they avoid the singular loci of their
	ambient spaces by \cite[Proposition~2.11]{przyjalkowski2020bounds}.
	Applying Lemma~\ref{ar:cotangent} to the morphism
	$f:\Sigma\to X$ in Lemma~\ref{ar:reduction} proves the assertion.
\end{proof}

The following lemma extends \cite[Lemma~7.2]{LTT}.
\begin{lemm}[Pullback along conic bundles]\label{ar:conic}
	Let $\pi:X\to B$ be a conic bundle over a smooth surface, and let
	$\Gamma\subset B$ be a smooth integral curve with $\Gamma^2>0$.
	If $S=\pi^{-1}(\Gamma)$ is integral, then $S$ does not satisfy
	\eqref{ar:assumption}.
\end{lemm}

\begin{proof}
	Suppose that $S$ satisfies \eqref{ar:assumption}, and let $D_\pi$
	denote the discriminant of $\pi$.
	Suppose first that $\Gamma\subset D_\pi$.
	Since $S$ is integral, the geometric generic fiber over $\Gamma$
	consists of two distinct lines.
	After resolution and Stein factorization, we obtain a morphism
	from $\Sigma$ to a smooth curve whose general fiber $R$ is a
	smooth rational curve mapping birationally onto a component of
	a singular fiber of $\pi$.
	Thus $f^*L\cdot R=1$, and \eqref{ar:degree} gives
	$(-K_\Sigma)\cdot R=1$.
	On the other hand, since $R^2=0$, adjunction gives
	$(-K_\Sigma)\cdot R=2$, a contradiction.
	
	Now suppose that $\Gamma\not\subset D_\pi$. We claim that $S$ is normal.
	Since $S$ is Cartier in $X$, it satisfies Serre's condition $S_2$.
	Any curve contained in $\operatorname{Sing}(S)$ must
	lie in a fiber over the finite set $\Gamma\cap D_\pi$.
	At a general point of either component of a reduced
	singular conic, the inverse image is smooth. Thus such a curve would have
	to be the support $\ell$ of a double-line fiber.
	
	Here the local argument of \cite[Lemma~7.2]{LTT} applies over any smooth
	base surface. Choose local parameters $u,v$ with $\Gamma=(u=0)$ and write
	the conic bundle as
	$F(u,v;t_0,t_1,t_2)=0$ in a locally trivial $\bP^2$-bundle.
	If $S$ were singular along $\ell$, then $F_v$ and all $F_{t_i}$ would
	vanish there.
	However, $F_u|_\ell$ is a section of $\mathcal O_\ell(2)$,
	so it vanishes at some point $p\in\ell$.
	At $p$, all partial derivatives of $F$ vanish, contradicting
	the smoothness of $X$.
	Thus $\operatorname{Sing}(S)$ contains no curve, and $S$ satisfies
	Serre's condition $R_1$.
	Since $S$ also satisfies $S_2$, it is normal by Serre's criterion.
	
	Lemma~\ref{ar:conductor} now gives $LS^2=0$.
	On the other hand, since $S=\pi^*\Gamma$ and a general fiber
	$F$ of $\pi$ satisfies $L\cdot F=2$, the projection formula gives
	\[
	LS^2=(L\cdot F)\Gamma^2=2\Gamma^2>0,
	\]
	a contradiction.
\end{proof}

For families $2.1$ and $2.3$, the following lemma shows that
the blow-downs to $V_1$ and $V_2$, respectively, preserve
adjoint rigid surfaces satisfying~\eqref{ar:assumption}
and dominating the base of the del Pezzo fibration.

\begin{lemm}[Descent along a half-anticanonical blow-up]
	\label{ar:blowdown}
	Let $\pi:X\to Y$ be the blow-up of a smooth Fano threefold $Y$
	along a smooth connected complete intersection
	$C=D_1\cap D_2$, where $D_1,D_2\in|H|$ and $-K_Y\sim2H$.
	Let $\varphi:X\to\bP^1$ be the morphism induced by the
	pencil $\langle D_1,D_2\rangle$.
	Assume that $X$ is Fano, and let $E$ be the exceptional divisor.
	
	Suppose that an integral surface $S\ne E$ satisfies
	\eqref{ar:assumption} and dominates $\bP^1$ under $\varphi$.
	Then $D=\pi(S)$ satisfies $a(D,-K_Y)=1$ and is adjoint rigid
	with respect to $-K_Y$.
\end{lemm}
\[
\begin{tikzcd}[column sep=large, row sep=large]
	\bP^1
	& X \arrow[l, "\varphi"']
	\arrow[r, "\pi"]
	& Y \\
	\Sigma \arrow[ur, "f"]
	\arrow[r, "p"]
	\arrow[dr, "\bar h"']
	& S \arrow[u, hook]
	\arrow[r, "\pi|_S"]
	& D \arrow[u, hook] \\
	& D^\nu \arrow[ur, "\nu_D"']
\end{tikzcd}
\]
\begin{proof}
	Retain the notation of Lemma~\ref{ar:reduction}, and set
	$p=\nu\circ r:\Sigma\to S$ and $h=\pi\circ f$.
	
	By adjunction, $K_C\sim0$, so $C$ is an elliptic curve.
	Put $M=\pi^*H$, and let $F=M-E$ be the fiber class of $\varphi$.
	The blow-up formula gives
	\[
	L=-K_X=2M-E=M+F,
	\qquad L-E=2F.
	\]
	Since $S$ dominates $\bP^1$, its intersection with a general
	fiber is a nonzero effective curve. Hence $LSF>0$, and therefore
	\begin{equation}\label{ar:exceptional-degree}
		LSE=d-2LSF<d,
	\end{equation}
	where $d=L^2S=K_\Sigma^2$.
	
	Since $S\ne E$, the morphism $h$ induces a birational morphism
	$\bar h:\Sigma\to D^\nu$, where
	$\nu_D:D^\nu\to D=\pi(S)$ is the normalization.
	We claim that $f^*E$ is $\bar h$-exceptional.
	Otherwise, since $\nu_D$ is finite, some irreducible component
	$R$ of $f^*E$ would dominate $C$ under $h$.
	Then $f(R)$ is a curve, so the ampleness of $L$ and the nefness
	of $-K_\Sigma\sim f^*L$ give
	\[
	0<(-K_\Sigma)\cdot R
	\le f^*L\cdot f^*E
	=LSE<d.
	\]
	By Lemma~\ref{ar:numerics}(2), $R$ is a smooth rational curve,
	which cannot dominate the elliptic curve $C$.
	
	The blow-up formula and \eqref{ar:crepant} give
	\begin{equation}\label{ar:exceptional-adjoint}
		K_\Sigma+h^*(-K_Y)\sim f^*E.
	\end{equation}
	Since $f^*E$ is effective and $\bar h$-exceptional,
	Iitaka's covering theorem
	\cite[Theorem~10.5]{iitaka1982algebraic} yields
	\begin{equation}\label{eq-kappa}
		\kappa\bigl(\Sigma,K_\Sigma+h^*(-K_Y)\bigr)=0.
	\end{equation}
	
	The effectiveness of $f^*E$ and \eqref{ar:exceptional-adjoint}
	give $a(D,-K_Y)\le1$.
	If $a(D,-K_Y)<1$, the bigness of $h^*(-K_Y)$ would imply
	that $K_\Sigma+h^*(-K_Y)$ is big, contradicting \eqref{eq-kappa}.
	Thus $a(D,-K_Y)=1$, and \eqref{eq-kappa} proves adjoint rigidity.
\end{proof}

The following lemma will be used to exclude the double-cover cases
over $\bP^1\times\bP^2$ and the blow-up of $\bP^3$
at a point.

\begin{lemm}\label{ar:double}
	Let $Y$ be a smooth projective threefold, and let $M$ be a Cartier
	divisor on $Y$. Let $q:X\to Y$ be the double cover defined by a
	section of $\mathcal O_Y(2M)$ whose zero divisor $B$ is smooth.
	Suppose that $D\subset Y$ is a smooth integral surface and that
	$S=q^{-1}(D)$ is integral.
	Let $\nu:T\to S$ be the normalization, and let $g:T\to X$
	be its composition with the inclusion.
	If $\omega_T\simeq g^*\omega_X$, then
	\begin{equation}\label{ar:double-obstruction}
		D^2(B-D)=0.
	\end{equation}
\end{lemm}
\begin{proof}
	Write
	\[
	B|_D=2R+B_{\mathrm{odd}},
	\]
	where $R$ is effective and $B_{\mathrm{odd}}$ is reduced.
	Let $\Gamma$ be an irreducible component of $R$.
	Let $\mathcal I$, $\mathcal I_{B}$, and $\mathcal I_D$
	be the ideal sheaves of $\Gamma$, $B$, and $D$ in $Y$,
	respectively.
	Since $B|_D\ge 2\Gamma$, we have
	\begin{equation}\label{eq:IB0}
		\mathcal I_{B}\subset\mathcal I^2+\mathcal I_D.
	\end{equation}
	The conormal exact sequences for the chains
	\(\Gamma\subset B\subset Y\) and \(\Gamma\subset D\subset Y\) give
	\[
	\begin{tikzcd}[column sep=small]
		0 \arrow[r]
		& \mathcal O_\Gamma(-B) \arrow[r] \arrow[d, "\varphi"]
		& \mathcal I/\mathcal I^2 \arrow[r] \arrow[d, equal]
		& \mathcal I/(\mathcal I^2+\mathcal I_{B})
		\arrow[r] \arrow[d, "\psi"]
		& 0 \\
		0 \arrow[r]
		& \mathcal O_\Gamma(-D) \arrow[r]
		& \mathcal I/\mathcal I^2 \arrow[r]
		& \mathcal I/(\mathcal I^2+\mathcal I_D)
		\arrow[r]
		& 0.
	\end{tikzcd}
	\]
	The inclusions \(\Gamma\subset B\subset Y\) and
	\(\Gamma\subset D\subset Y\) are regular embeddings. Thus the rows
	are exact. Since \(\Gamma\) is Cartier in both \(B\) and \(D\),
	the right-hand terms are the conormal line bundles
	$\mathcal O_B(-\Gamma)|_\Gamma$ and
	$\mathcal O_D(-\Gamma)|_\Gamma$, respectively.
	The map $\psi$ is the natural quotient map by \eqref{eq:IB0}, so it is an isomorphism.
	It follows that $\varphi$ is also an isomorphism, and hence
	$(B-D)\cdot\Gamma=0$.
	Summing with the multiplicities in $R$, we obtain
	\[
	(B-D)|_D\cdot R=0.
	\]
	
	It remains to show that $R\equiv D|_D$.
	Put $p=q|_S\circ\nu:T\to D$.
	Removing the square factor gives a finite birational morphism
	$T'\to S$, with $T'\to D$ the double cover defined by a section of
	$\mathcal O_D(2(M|_D-R))$.
	The branch divisor $B_{\mathrm{odd}}$ is reduced, so $T^\prime$ is regular
	in codimension one; it is also a hypersurface, hence Cohen--Macaulay.
	Thus $T^\prime$ identifies with the normalization $T$.
	The double-cover formula and adjunction on $Y$ give
	\[
	\begin{aligned}
		K_T
		&\sim p^*(K_D+M|_D-R)\\
		&\sim p^*\bigl((K_Y+M)|_D\bigr)+p^*(D|_D-R).
	\end{aligned}
	\]
	On the other hand, $K_X\sim q^*(K_Y+M)$ gives
	\[
	g^*K_X\sim p^*\bigl((K_Y+M)|_D\bigr).
	\]
	Subtracting, we obtain
	\[
	p^*(D|_D-R)\sim K_T-g^*K_X\sim0.
	\]
	We thus have $R\equiv D|_D$ since $p$ is a finite surjection.
	The conclusion follows.
\end{proof}

\subsection{Numerical reduction}
\label{ar:classification-section}

We use the Mori--Mukai numbering $\rho.n$ for deformation families.
Parts of the proof below follow the argument of
\cite[Theorem~5.3]{beheshti2022moduli}.
Here we treat individual adjoint rigid surfaces with
$a(S,-K_X)=1$, but not a covering family of them,
and we also allow $-K_X$ to be non-very-ample.

\begin{prop}[Numerical reduction]\label{ar:classification}
	Suppose that $X$ is not a product, that $\rho(X)\ge2$, and that $S$
	satisfies \eqref{ar:assumption}. Then one of the following holds.
	\begin{enumerate}[label={\rm(\arabic*)},ref=\arabic*]
		\item\label{ar:case-fibre}
		The surface $S$ is a reduced normal fiber of a del Pezzo
		fibration $X\to\bP^1$.
		\item\label{ar:case-conic}
		There is a conic bundle $\pi:X\to B$ over a smooth surface and a smooth
		integral curve $\Gamma\subset B$, with $\Gamma^2>0$, such that
		$S=\pi^{-1}(\Gamma)$.
		\item\label{ar:case-elliptic}
		The threefold is $X=\operatorname{Bl}_C V_j$, where $j=1$ or $2$ and
		$C$ is the smooth complete intersection of two members of the
		ample generator $|H|$ on $V_j$.
		Let \(\pi:X\to V_j\) be the blow-up and \(E\) its exceptional divisor.
		With $M=\pi^*H$ and $F=M-E$, one has
		$S\sim uM+vF$ for integers $u>0$, $v\ge0$.
		\item\label{ar:case-22}
		The threefold is a double cover $q:X\to\bP^1\times\bP^2$
		branched in bidegree $(2,4)$, and
		\[
		S\sim q^*D,\qquad
		D\in\{B,A+B,2A+B,2B\},
		\]
		where $A,B$ are the two hyperplane classes on the base.
		\item\label{ar:case-28}
		The threefold is a double cover
		$q:X\to Y=\operatorname{Bl}_p\bP^3$ branched in $|4h-2e|$, where $h$ is
		the pullback of a hyperplane and $e$ is the exceptional divisor, and
		$S\sim q^*h$.
	\end{enumerate}
\end{prop}
\begin{proof}
	By Lemmas~\ref{ar:reduction} and \ref{ar:conductor}, $S$ is nef and
	\[
	1\le L^2S\le9.
	\]
	Suppose that $S\sim mF$, where $F$ is the full fiber class of a
	del Pezzo fibration $\pi:X\to\mathbf P^1$ and $m\ge1$.
	Then $m=1$ since $S$ is integral.
	Moreover, $LS^2=0$, so $S$ is normal by
	Lemma~\ref{ar:conductor}. Thus case~\ref{ar:case-fibre} holds.
	
	It remains to determine the possible classes of $S$ that are not
	positive multiples of del Pezzo fiber classes. For this, we use
	the Mori--Mukai models and their intersection data \cite{MM83},
	together with the descriptions of the nef cones in
	\cite[\S10]{Fujita2016} and
	\cite[\S\S8--12]{burke2022geometric}.
	
	\par\medskip
	\noindent\textbf{(Picard number two).}
	Let $L_1,L_2$ be the primitive generators of the extremal rays of
	$\operatorname{Nef}(X)$.
	By \cite[Theorem~5.1]{MM83}, they form a $\bZ$-basis of
	$\operatorname{Pic}(X)$.
	Thus
	\[
	S\sim aL_1+bL_2,\qquad
	a,b\in\bZ_{\ge0},\qquad
	aL^2L_1+bL^2L_2\le9.
	\]
	For each family, the two generator degrees therefore determine
	a finite set of candidate divisor classes.
	There are nine primitive families:
	\[
	2.2,\quad2.6,\quad2.8,\quad2.18,\quad2.24,\quad
	2.32,\quad2.34,\quad2.35,\quad2.36,
	\]
	see \cite[Theorem~1.7]{MM83}.
	We first treat imprimitive families, which admit blow-up
	presentations along smooth curves.
	
	\par\smallskip
	\noindent\textbf{(Imprimitive families).}
	Write $\pi:X\to Y$ for the blow-up of a smooth curve $C$ of
	genus $\gamma$, let $E$ be its exceptional divisor, and put
	$M=\pi^*H$, where $H$ is the ample generator of
	$\operatorname{Pic}(Y)$.
	We have
	\begin{equation}\label{ar:blowup-numbers}
		L^2M=(-K_Y)^2H-H\cdot C,\qquad
		L^2E=(-K_Y)\cdot C-2\gamma+2.
	\end{equation}
	
	See \cite[Section 4.4]{Crowley2020} for the description of each family.
	In families $2.1$ and $2.3$, the nef generators are $M$ and
	$F=M-E$, with
	\[
	(L^2M,L^2F)=(3,1)\quad\text{and}\quad(6,2),
	\]
	respectively. Here $Y=V_1$ or $V_2$, and $F$ is a del Pezzo
	fiber class. Every nonfiber candidate therefore has the form
	$S\sim uM+vF$, with $u>0$ and $v\ge0$, giving
	case~\eqref{ar:case-elliptic}.
	
	In families $2.11$, $2.13$, $2.16$, and $2.20$, the class $M$ is excluded
	by the degree bound.
	In families $2.4$, $2.5$, and $2.9$, the class $M$ satisfies
	the degree bound, but $S\sim M$ is excluded by \eqref{ar:parity},
	since $M^2S+LMS=5,9,5$, respectively.

	In families $2.4$ and $2.5$, the other nef generator is a del Pezzo
	fiber class.
	In families $2.9$, $2.11$, $2.13$, $2.16$, and $2.20$, the
	other nef generator $G$ is the pullback of a line under a
	conic bundle to $\bP^2$. The generator degrees are
	\[
	(L^2M,L^2G)
	=(9,7),\ (11,7),\ (12,8),\ (14,8),\ (17,9),
	\]
	respectively. Apart from the class $M$ in family $2.9$,
	which was excluded above, the degree bound leaves only $G$.
	This gives case~\eqref{ar:case-conic}.
	
	Among all remaining imprimitive families, the degree bound
	leaves only multiples of del Pezzo fiber classes in
	$2.7$, $2.10$, $2.14$, $2.25$, $2.29$, and $2.33$,
	giving case~\eqref{ar:case-fibre}.
	
	\par\smallskip
	\noindent\textbf{(Primitive families).}
	We now treat the nine primitive families.
	The nef generators below are also recorded in the
	calculations of \cite[\S7]{LTT}.
	
	For the double cover $q:X\to\bP^1\times\bP^2$
	in family $2.2$, let $A,B$ denote the hyperplane classes on
	the two factors.
	The nef generators $q^*A,q^*B$ have degrees $2,4$, and
	$q^*A$ is a del Pezzo fiber class.
	The inequality $2a+4b\le9$, with $a\ge0$ and $b>0$, leaves
	precisely
	\[
	q^*B,\qquad q^*(A+B),\qquad
	q^*(2A+B),\qquad q^*(2B),
	\]
	giving \eqref{ar:case-22}.
	
	In families $2.6$, $2.8$, $2.18$, and $2.24$, the generator
	degrees $(L^2L_1,L^2L_2)$ are
	\[
	(6,6),\qquad(8,6),\qquad(8,8),\qquad(9,12),
	\]
	respectively.
	The degree bound leaves both generators in the first three
	families and only $L_1$ in family $2.24$.
	
	In family $2.8$, let
	$q:X\to Y=\operatorname{Bl}_p\bP^3$ be the double cover.
	Write $h$ for the pullback of the hyperplane class under
	$Y\to\bP^3$ and $e$ for the exceptional divisor.
	The two nef generators are
	\[
	L_1=q^*h,\qquad L_2=q^*(h-e).
	\]
	The class $q^*h$ gives case~\eqref{ar:case-28}, whereas
	$q^*(h-e)$ is the pullback of a line under the conic bundle
	$X\to\bP^2$ induced by projection from $p$.
	In family $2.18$, one generator is a del Pezzo fiber class,
	giving case~\eqref{ar:case-fibre}.
	All remaining candidates are pullbacks of lines along conic bundles,
	giving case~\eqref{ar:case-conic}.
	
	In families $2.32$, $2.35$, and $2.36$, both generators have degree
	above nine, so there are no candidates.
	Finally, $2.34$ is the product $\bP^1\times\bP^2$,
	excluded by hypothesis.
	This completes the rank-two analysis.
	
	\par\medskip
	\noindent\textbf{(Picard number at least three).}
	Choose an integral basis $D_1,\ldots,D_\rho$ of
	$\operatorname{Pic}(X)$ and generators
	$\gamma_1,\ldots,\gamma_s$ of the Mori cone.
	Writing $S\sim\sum_i a_iD_i$, the numerical conditions are
	\[
	a_i\in\mathbb Z,\qquad
	\sum_i a_i(D_i\cdot\gamma_j)\geq0
	\quad(1\leq j\leq s),\qquad
	1\leq\sum_i a_iL^2D_i\leq9.
	\]
	Since $L^2$ is strictly positive on every nonzero nef class,
	these inequalities leave finitely many integral divisor classes.
	We also impose the parity condition~\eqref{ar:parity}.
	
	For $\rho(X)=3$, apart from positive multiples of del Pezzo
	fiber classes, the degree bound leaves only pullbacks of curve
	classes under conic bundles, as in case~\eqref{ar:case-conic}.
	The corresponding curve classes on the bases are $(1,1)$ on
	$\bP^1\times\bP^1$ in families $3.1$ and $3.2$;
	lines on $\bP^2$ in families $3.3$, $3.7$, $3.9$,
	$3.11$, and $3.14$;
	and the pullback $h$ of the line class on $\bP^2$ to
	$\operatorname{Bl}_p\bP^2$ in families $3.4$ and $3.8$.
	
	For nonproduct families of Picard number four or five,
	only del Pezzo fiber classes remain, giving
	case~\eqref{ar:case-fibre}.
	All families of Picard number at least six are products
	and are excluded by hypothesis.
	
	To finish, we verify that the conic-bundle classes above give the
	asserted description for every integral member.
	Let $\pi:X\to B$ be one of these conic bundles and write
	$S\sim\pi^*\Gamma_0$. Since $\pi_*\mathcal O_X=\mathcal O_B$,
	the projection formula gives
	\[
	H^0(X,\mathcal O_X(\pi^*\Gamma_0))
	=H^0(B,\mathcal O_B(\Gamma_0)).
	\]
	Thus $S=\pi^{-1}(\Gamma)$ scheme-theoretically for some
	$\Gamma\in|\Gamma_0|$, and integrality of $S$ forces $\Gamma$
	to be integral. In the listed systems, every integral member is
	smooth: it is a plane line, a $(1,1)$-curve on
	$\bP^1\times\bP^1$, or the strict transform of a line
	avoiding $p$ on $\operatorname{Bl}_p\bP^2$.
	Their self-intersections are $1,2,1$, respectively.
	This proves case~\eqref{ar:case-conic} and completes the proof of the proposition.
\end{proof}

\subsection{Proof of the main theorem}

\begin{prop}\label{ar:elliptic}
	Case~\eqref{ar:case-elliptic} of Proposition~\ref{ar:classification} does
	not occur.
\end{prop}

\begin{proof}
	Retain the notation of case~\eqref{ar:case-elliptic}.
	Since $u>0$, the surface $S$ dominates the base $\bP^1$
	of the del Pezzo fibration.
	Lemma~\ref{ar:blowdown} then implies that $\pi(S)\subset V_j$
	is adjoint rigid with $a$-invariant one, contrary to
	Proposition~\ref{ar:rank1-exclusion}.
\end{proof}

\begin{prop}[The double-cover cases]\label{ar:double-exclusion}
	Cases~\eqref{ar:case-22} and \eqref{ar:case-28} of
	Proposition~\ref{ar:classification} do not occur.
\end{prop}
\begin{proof}
	Retain the notation of Proposition~\ref{ar:classification}.
	Let $M$ be the line bundle defining the double cover, so that
	\[
	q_*\mathcal O_X
	=\mathcal O_Y\oplus\mathcal O_Y(-M).
	\]
	Here $M=A+2B$ on $\bP^1\times\bP^2$, and
	$M=2h-e$ on $\operatorname{Bl}_p\bP^3$.
	For each candidate $S\in|q^*D|$ (where we write $D=h$ in Case \eqref{ar:case-28}), the projection formula gives
	\[
	H^0\bigl(X,\mathcal O_X(q^*D)\bigr)
	=
	H^0\bigl(Y,\mathcal O_Y(D)\bigr)
	\oplus
	H^0\bigl(Y,\mathcal O_Y(D-M)\bigr).
	\]
	The second summand vanishes: on
	$\bP^1\times\bP^2$, the classes $D-M$ are
	$-A-B$, $-B$, $A-B$, and $-A$, respectively;
	on $\operatorname{Bl}_p\bP^3$, the class is $-h+e$.
	Thus the section defining $S$ is the pullback of a section
	$s\in H^0(Y,\mathcal O_Y(D))$.
	Setting $D=(s=0)$, we obtain
	$S=q^{-1}(D)$ scheme-theoretically.
	Since $q$ is faithfully flat and $S$ is integral, $D$ is integral.
	
	We claim that $D$ is smooth in every case.
	On $\bP^1\times\bP^2$, an integral member of $|aA+B|$,
	$a=0,1,2$, is a $\bP^1$-bundle over $\bP^1$,
	while an integral member of $|2B|$ is the product of
	$\bP^1$ with a smooth plane conic.
	On $\operatorname{Bl}_p\bP^3$, an integral member of $|h|$
	is the strict transform of a plane in $\bP^3$ not containing $p$,
	and hence is isomorphic to $\bP^2$.
	
	Lemma~\ref{ar:double} therefore gives $D^2(B_0-D)=0$,
	where $B_0$ is the branch divisor.
	On $\bP^1\times\bP^2$, we have
	\[
	\begin{aligned}
		D^2(B_0-D)
		&=(aA+B)^2\bigl((2-a)A+3B\bigr)
		=5a+2
		&&\text{if }D\sim aA+B,\quad a=0,1,2,\\
		D^2(B_0-D)
		&=(2B)^2(2A+2B)=8
		&&\text{if }D\sim2B.
	\end{aligned}
	\]
	On $\operatorname{Bl}_p\bP^3$, we have
	$B_0\sim4h-2e$ and $D\sim h$, hence
	\[
	D^2(B_0-D)=h^2(3h-2e)=3.
	\]
	All these values are nonzero, giving a contradiction.
\end{proof}

We finally give a proof of the main theorem of this section.
\begin{proof}[Proof of Theorem~\ref{ar:main}]
	Assume that $X$ is not a product. Proposition~\ref{ar:rank1-exclusion}
	excludes Picard number one. For higher Picard number,
	Proposition~\ref{ar:classification} applies.
	Cases~\eqref{ar:case-conic}--\eqref{ar:case-28} are excluded by
	Lemma~\ref{ar:conic},
	Proposition~\ref{ar:elliptic}, and
	Proposition~\ref{ar:double-exclusion}.
	Hence Case~\eqref{ar:case-fibre} holds, and $S$ is a reduced normal
	fiber of a del Pezzo fibration.
	The surface $S$ has Du Val singularities by Lemma~\ref{ar:reduction}.
\end{proof}

\section{$b$-exceptional sets and del Pezzo fibrations}
\label{sec-bsuface}

Let $X$ be a smooth Fano threefold over a number field $F$, and put
$L=-K_X$. We write $Z(X)$ for the geometric exceptional set of
Definition~\ref{defi-ges}.
We write $Z_a(X)$ for the union of rational-point images of
resolutions of subvarieties with higher $a$-invariant, as defined
in Section~\ref{sec-3}.

Theorem~\ref{ar:main} identifies the adjoint rigid surfaces with
$a$-invariant one on a geometrically non-product threefold as fibers
of del Pezzo fibrations. We first show that the fibers with worse
singularities have higher $a$-invariant.

\subsection{Degenerate fibers and the $a$-exceptional locus}
\label{sec-degenerate-fibers}

\begin{prop}\label{prop-dp-bad-fibers}
	Let $\pi:X\to B$ be a del Pezzo fibration on a smooth Fano threefold
	over an algebraically closed field of characteristic zero, and put
	$L=-K_X$. Then the following hold.
	\begin{enumerate}[label={\rm(\arabic*)}]
		\item Let $S$ be an integral fiber. If either $S$
		is non-normal or has non-canonical singularities,
		then $a(S,L|_S)>1$.
		\label{prop-dp-bad-fibers-1}
		\item Every irreducible component $S$ of a reducible fiber satisfies
		$a(S,L|_S)>1$.
		\label{prop-dp-bad-fibers-2}
		\item Every fiber is reduced.
		\label{prop-dp-bad-fibers-3}
	\end{enumerate}
	Consequently, an integral fiber $S$ has $a(S,L|_S)=1$ if and only if it has
	at most Du Val singularities. In this case $(S,L|_S)$ is adjoint
	rigid. Every irreducible component of any other fiber has
	$a$-invariant greater than one.
\end{prop}

\begin{proof}
	Every fiber of $\pi$ is reduced by
	\cite[Proposition~1.4 and Lemma~2.8]{Fujita1990}.
	Let $S$ be an irreducible component of a fiber.
	Let $\nu:T\to S$ be its normalization, let $r:\widetilde S\to T$
	be a resolution, and set $\mu=\nu\circ r$ and $H=\nu^*(L|_S)$.
	Write $\Delta$ for the effective conductor divisor on $T$.
	Adjunction and the normalization formula
	\cite[Proposition~2.3]{Reid94}
	\[
	K_T+\Delta\sim\nu^*K_S,
	\qquad K_S+L|_S\sim S|_S
	\]
	imply
	\[
	\begin{aligned}
		K_T+H
		\sim \nu^*(K_S+L|_S)-\Delta
		\sim \nu^*(S|_S)-\Delta.
	\end{aligned}
	\]
	By the projection formula,
	\begin{align}\label{eq-fiber-conductor-intersection}
			\bigl(K_{\widetilde S}+\mu^*(L|_S)\bigr)
			\cdot\mu^*(L|_S)
			&=(K_{\widetilde S}+r^*H)\cdot r^*H\notag\\
			&=(K_T+H)\cdot H.\notag\\
			&=LS^2-H\cdot\Delta.	
	\end{align}
	Here $H$ is ample and $\mu^*(L|_S)$ is nef.
	% Since $S$ is
	% Cohen--Macaulay, nonnormality implies failure of normality in
	% codimension one, so $\Delta\ne0$ whenever $S$ is nonnormal.
	
	Suppose first that $S$ is an integral fiber. Then
	$\mathcal O_S(S)\simeq\mathcal O_S$ and $LS^2=0$.
	If $S$ is non-normal, the right-hand side of
	\eqref{eq-fiber-conductor-intersection} is negative. Thus
	$K_{\widetilde S}+\mu^*(L|_S)$ is not pseudo-effective, and
	therefore $a(S,L|_S)>1$, proving~\ref{prop-dp-bad-fibers-1}.
	
	If $S$ is normal, let $\mu$ be its minimal resolution and write
	\[
	K_{\widetilde S}+\mu^*(L|_S)
	\sim K_{\widetilde S}-\mu^*K_S
	=\sum_i a_iE_i,
	\]
	where the $a_i$ are the discrepancies.
	If $S$ is not canonical, then some $a_i<0$, while $a_i\leq0$
	for every $i$. Hence
	\(
	K_{\widetilde S}+\mu^*(L|_S)\sim -E
	\)
	for some nonzero effective divisor $E$.
	Its intersection with an ample divisor is negative, so it is not
	pseudo-effective. Thus $a(S,L|_S)>1$.
	
	If $S$ has at most canonical singularities, its minimal resolution
	is crepant, so
	\[
	K_{\widetilde S}+\mu^*(L|_S)\sim0.
	\]
	Hence $a(S,L|_S)=1$ and $(S,L|_S)$ is adjoint rigid.

	For a reducible fiber, write $X_b=\sum_{i=1}^k S_i$,
	with $k\geq2$.
	Since $\mathcal O_X(X_b)|_{S_i}\simeq\mathcal O_{S_i}$,
	\[
	LS_i^2=-\sum_{j\ne i}LS_iS_j<0,
	\]
	where the strict inequality follows from connectedness of
	the fiber and ampleness of $L$.
	Equation~\eqref{eq-fiber-conductor-intersection} therefore gives
	$a(S_i,L|_{S_i})>1$ for every $i$, proving~\ref{prop-dp-bad-fibers-2}.
\end{proof}

\subsection{The conic and del Pezzo contributions}
\label{sec-b-reduction}
\label{sec-moving-adjoint-rigid}

An anticanonical conic is a geometrically integral curve
$C\subset X$ whose normalization is geometrically rational
and for which $L\cdot C=2$.
The following proposition describes the geometric exceptional
set of geometrically non-product smooth Fano threefolds.

\begin{prop}\label{prop-breaking-map-reduction}
	Let $X$ be a smooth Fano threefold over a number field $F$
	such that $X_{\overline F}$ is not a product.
	Put $L=-K_X$ and $r=\rho_F(X)$.
	Then $Z(X)$ is the union of $Z_a(X)$ and the following
	rational-point images, with all source varieties taken
	to be smooth, projective, and geometrically integral:
	\begin{enumerate}[label={\rm(\arabic*)}]
		\item
		If $r=1$, the images $f(Y(F))$ of thin maps with
		$a(Y,f^*L)=1$ whose general geometric adjoint Iitaka
		fibers map birationally onto anticanonical conics.
		
		\item
		For each geometrically integral del Pezzo fiber $S$
		over an $F$-rational point with at worst Du Val singularities,
		let $\mu:\widetilde S\to S$ be its minimal resolution.
		Include $\mu(\widetilde S(F))$ if $\rho_F(\widetilde S)\geq r$.
		Otherwise, include $g(Y(F))$ for every adjoint rigid
		$a$-cover $g:Y\to S$ satisfying
		\[
		b(F,Y,g^*(L|_S))\geq r.
		\]
		
		\item
		The full images $f(Y(F))$ of smooth projective models
		of dominant base changes $X\times_B T\to X$, where
		$\pi:X\to B$ is a del Pezzo fibration and
		$T\to B$ is a finite cover of degree at least two
		by a smooth projective geometrically integral curve,
		provided that
		\[
		\rho_{F(T)}(X_{F(T)})\geq r.
		\]
	\end{enumerate}
\end{prop}

\begin{proof}
	Outside the higher-$a$ contribution,
	\cite[Lemmas~4.3 and~4.9]{LST} and
	\cite[Theorem~5.5]{beheshti2022moduli} reduce the classification
	to adjoint Iitaka fibers of dimension one or two.
	The first case gives {\rm(1)}, with $b=1$ by
	\cite[Lemma~4.15(1)]{LST}.
	The second gives {\rm(2)} and {\rm(3)} by
	Theorem~\ref{ar:main} and
	Proposition~\ref{prop-breaking-surfaces-move}.
	For smooth del Pezzo fibers there are no adjoint rigid
	$a$-covers by \cite[Theorem~6.2]{LTDuke}.
	In {\rm(2)}, Lang--Nishimura places every cover image
	inside $\mu(\widetilde S(F))$ whenever the resolution
	is already breaking.
	The criterion in {\rm(3)} follows from
	Lemma~\ref{lemm-dp-base-change-b}.
	Finally, resolutions of geometrically integral non-normal
	or non-canonical fibers, and of geometrically
	integral components of reducible fibers, already contribute
	to $Z_a(X)$ by Proposition~\ref{prop-dp-bad-fibers}.
\end{proof}

\begin{rema}\label{rema-ges-density-reduction}
	In the setting of Proposition~\ref{prop-breaking-map-reduction},
	Zariski density of $Z(X)$ can only arise from the conic
	contribution in {\rm(1)}, or from smooth breaking del Pezzo
	fibers and the dominant breaking base changes in {\rm(3)}.
	
	Indeed, $Z_a(X)$ is contained in a proper closed subset by
	Section~\ref{sec-3}. There are only finitely many del Pezzo
	fibrations on $X$, up to isomorphism of the base, since their
	contractions correspond to faces of the rational polyhedral
	Mori cone. The union of their singular fibers is therefore
	also a proper closed subset. Smooth del Pezzo surfaces admit
	no adjoint rigid $a$-covers, so their contribution in {\rm(2)}
	consists precisely of $S(F)$ for smooth fibers satisfying
	$\rho_F(S)\geq\rho_F(X)$.
	
	The conic contribution occurs only when $\rho_F(X)=1$.
	In this case no del Pezzo fibration is defined over $F$.
	When $\rho_F(X)>1$, only the del Pezzo contribution can
	account for Zariski density.
\end{rema}

The following proposition records the base-change argument used in
\cite[proof of Theorem~5.3]{beheshti2022moduli} and
\cite[proof of Theorem~3.15, Steps~1--2]{gao2024geometric}.
\begin{prop}[Surface fibers and base change]
	\label{prop-breaking-surfaces-move}
	Let $\pi:X\to B$ be a del Pezzo fibration from a smooth
	Fano variety over $F$ to a smooth projective variety.
	Let $f:Y\to X$ be an $a$-cover, and let
	$p:Y\dashrightarrow T$ be its adjoint Iitaka fibration.
	Suppose that $f$ maps a general geometric fiber of $p$ onto
	a general geometric fiber of $\pi$.
	Then, after replacing $T$ by a smooth birational model,
	there is a generically finite morphism $T\to B$ such that
	$(f,p)$ induces a birational map from $Y$ to the main component
	of $X\times_B T$.
\end{prop}

\begin{proof}
	Resolve $p$ and let $E$ be a general geometric fiber.
	By \cite[Lemma~4.3]{LST}, the pair
	$(E,f^*(-K_X)|_E)$ is adjoint rigid with $a$-invariant one.
	Its image $S=f(E)$ is a smooth del Pezzo fiber of $\pi$, with
	$(-K_X)|_S\sim-K_S$. Since a smooth del Pezzo surface admits
	no adjoint rigid $a$-cover, the map $E\to S$ is birational
	by \cite[Theorem~6.2]{LTDuke}.
	
	Since $\pi\circ f$ is constant on general fibers of $p$,
	the rigidity lemma \cite[Lemma~1.15]{Deb} yields a dominant
	rational map $g:T\dashrightarrow B$ such that
	$\pi\circ f=g\circ p$.
	Since $f$ is generically finite and $\dim E=\dim S=2$,
	the map $g$ is generically finite.
	After resolving $g$, the induced map
	$Y\dashrightarrow X\times_B T$ is birational onto the main
	component, because it is birational on general geometric fibers
	over $T$.
\end{proof}

To determine which base changes are breaking, we express
their $b$-invariants in terms of the Picard ranks of their
generic fibers.
\begin{lemm}[The generic-fiber formula]
	\label{lemm-dp-base-change-b}
	Let $X$ be a smooth Fano variety over $F$, and let
	$\pi:X\to B$ be a del Pezzo fibration over a smooth projective
	variety. Let $g:T\to B$ be a dominant generically finite morphism,
	where $T$ is smooth, projective, and geometrically integral.
	Choose a smooth projective model $Y$ of the main component of
	$X\times_B T$, with induced morphisms $f:Y\to X$ and $p:Y\to T$.
	Then
	\[
	a(Y,f^*(-K_X))=1.
	\]
	If $p$ is birationally equivalent to the Iitaka fibration
	associated with $K_Y+f^*(-K_X)$, then
	\begin{equation}\label{eq-lemm-dp-base-change-b}
		b(F,Y,f^*(-K_X))
		=
		\rho_{F(T)}(X_{F(T)}).
	\end{equation}
	If $\dim X=3$ and $\deg g\geq2$, then the assumption is satisfied so
	equation \eqref{eq-lemm-dp-base-change-b} holds.
\end{lemm}

\begin{proof}
	By birational invariance, we may choose $Y$ with generic fiber
	$S=X_{F(T)}$ over $T$. Put $D=K_Y+f^*(-K_X)$.
	Ramification gives $D\geq0$, hence $a(Y,f^*(-K_X))\leq1$.
	Since $f^*(-K_X)|_S\sim-K_S$, restriction to $S$ gives the
	reverse inequality. Thus $a(Y,f^*(-K_X))=1$.
	
	Suppose that $p$ is birationally the Iitaka fibration of $D$.
	For any prime divisor $E\subset Y$ dominating $T$ and any $c>0$,
	the equality $D|_S\sim0$ gives
	\[
	(D-cE)|_S\equiv-c(E|_S).
	\]
	Since $E|_S$ is nonzero and effective, this class has negative
	intersection with an ample divisor on $S$.
	Thus $(D-cE)|_S$ is not pseudoeffective, and neither is $D-cE$.
	Consequently, the subspace generated by horizontal divisor classes
	in \cite[Lemma~4.15(1)]{LST} is zero, and the formula there yields
	\[
	b(F,Y,f^*(-K_X))=\rho_{F(T)}(S).
	\]
	
	Finally, suppose that $\dim X=3$ and $\deg g\geq2$.
	The divisor $D$ is effective and vertical over the curve $T$,
	so $0\leq\kappa(Y,D)\leq1$.
	Since $f$ is an $a$-cover, \cite[Theorem~5.5]{beheshti2022moduli}
	excludes $\kappa(Y,D)=0$.
	Hence $\kappa(Y,D)=1$, and its Iitaka fibration is birationally $p$.
\end{proof}

Thus, for a Fano threefold, a nontrivial del Pezzo base change
is breaking precisely when its generic Picard rank is at
least $\rho_F(X)$.

\subsection{Breaking thin maps arising from del Pezzo fibrations}
\label{sec-breaking-dP}

Fix a del Pezzo fibration $\pi:X\to B$ of degree $d$, where
$X$ is a smooth Fano variety over $F$ and $B$ is a
positive-dimensional smooth projective variety.
We construct a finite collection $P_\pi$ of candidate
breaking thin maps.
Theorem~\ref{theo-breaking} shows, in any dimension, that every
dominant breaking map whose general geometric adjoint Iitaka
fiber maps onto a general fiber of $\pi$ factors rationally
through a member of $P_\pi$.
For Fano threefolds, Remark~\ref{rema-dp-candidates}
characterizes the breaking members of $P_\pi$ in terms of
their generic Picard ranks.

\begin{cons}[Candidate thin maps from del Pezzo fibrations]
	\label{cons-relative para-spaces}\label{cons-dP-fib}
	Restrict $\pi$ over an open subset $B^\circ$ on which it is smooth,
	write $X^\circ=\pi^{-1}(B^\circ)$, and fix a
	geometric point $\bar b\to B^\circ$. Put $\Lambda=\Pic(X_{\bar b})$.
	The arithmetic \'etale fundamental group
	$\pi_1^{\mathrm{\acute et}}(B^\circ,\bar b)$ acts on $\Lambda$,
	preserving the intersection pairing, the canonical class, and the nef cone.
	Define the finite invariant sets
	\[
	\mathscr E_{\bar b}=
	\{e\in\Lambda\mid e^2=-1,\ (-K_{X_{\bar b}})\cdot e=1\},
	\]
	\[
	\mathscr R_{\bar b}=
	\{u\in\Lambda\mid u^2=0,\ (-K_{X_{\bar b}})\cdot u=2\}.
	\]
	They parametrize $(-1)$-curves and conic-bundle classes, respectively.
	For $1\leq r\leq9-d$, also take the finite set
	\[
	\mathscr E_{\bar b}^{(r)}=
	\{\{e_1,\ldots,e_r\}\subset\mathscr E_{\bar b}\mid
	e_i\cdot e_j=0\text{ for }i\ne j\}.
	\]
	The associated finite \'etale schemes over $B^\circ$ are denoted by
	$\mathscr E^{(r)}$ and $\mathscr R$. Normalize $B$ in the function fields
	of their connected components to obtain finite covers
	$p_r:\mathscr E^{(r)}\to B$ and $p_{\mathrm{cb}}:\mathscr R\to B$.

	For each geometrically integral connected component $N$ of
	$\mathscr E^{(r)}$ or $\mathscr R$ with
	$\deg(p_N:N\to B)\geq2$, choose a smooth projective model
	$Y_N$ of $X\times_B N$ such that the induced rational map to $X$
	extends to a morphism $\tau_N:Y_N\to X$.
	Let $P_\pi$ be the finite collection of these morphisms.
\end{cons}

\begin{theo}\label{theo-breaking}
	Let $X$ be a smooth Fano variety over $F$ with a del Pezzo
	fibration $\pi:X\to B$ as above. Let $f:Y\to X$ be a dominant thin map
	such that
	\[
	a(Y,f^*(-K_X))=1,
	\qquad b(F,Y,f^*(-K_X))\geq\rho_F(X).
	\]
	Suppose that a general geometric fiber of its adjoint Iitaka fibration
	maps onto a general fiber of $\pi$. Then $f$ factors rationally over
	$X$ through a member of $P_\pi$.
\end{theo}

\begin{proof}
	Let $Y\dasharrow T$ be the adjoint Iitaka fibration.
	By Proposition~\ref{prop-breaking-surfaces-move}, replace $Y$ by a
	smooth model of $X\times_B T$. Write $\eta$ for the generic point of
	$B$ and $\eta_T$ for that of $T$. The restriction on divisor classes
	is surjective and has nonzero kernel, since it kills the pullback of
	an ample class on the positive-dimensional base.
	By Lemma~\ref{lemm-dp-base-change-b},
	\[
	\rho_{F(T)}(X_{\eta_T})
	=b(F,Y,f^*(-K_X))
	\geq \rho_F(X)
	>\rho_{F(B)}(X_\eta).
	\]
	
	Let $G$ and $H\subset G$ be the images of the absolute Galois groups
	of $F(B)$ and $F(T)$ in the automorphism group of the geometric
	Picard lattice. Their invariant subspaces have different dimensions.
	The cone of curves of the del Pezzo surface over $F(T)$ spans the
	larger space, so some extremal contraction is not defined over $F(B)$.
	Indeed, if all its extremal contraction classes were $G$-invariant,
	the whole $H$-invariant space would be contained in the $G$-invariant
	space, a contradiction. Its Picard rank is at least two, so this is
	not a contraction to a point.
	
	An extremal contraction of a smooth del Pezzo surface of Picard rank
	at least two either contracts an invariant collection of disjoint
	$(-1)$-curves or is a conic-bundle contraction. The new contraction
	therefore gives an $F(T)$-point of $\mathscr E^{(r)}$ or $\mathscr R$
	whose image belongs to a component $N$ of degree at least two over
	$B$. Equivalently, $F(N)$ embeds in $F(T)$ over $F(B)$.
	This is the required rational factorization $T\dashrightarrow N$,
	and hence the factorization of $f$ through $\tau_N$.
\end{proof}

\begin{lemm}\label{lemm-dp-complexity}
	Let $\pi:X\to\bP^1_F$ be a del Pezzo fibration of degree $d$
	on a smooth Fano threefold over a number field $F$.
	Suppose that either a smooth fiber over an $F$-rational point
	is breaking or $\pi$ admits a dominant breaking base change.
	Then
	\[
	\rho_F(X)\leq10-d.
	\]
	In particular, $d\leq8$.
\end{lemm}

\begin{proof}
	For a breaking smooth fiber $S$, adjunction gives
	\[
	\rho_F(X)\leq b(F,S,-K_X|_S)=\rho_F(S)\leq10-d.
	\]
	For a breaking base change $f:Y\to X$ associated with a finite
	cover $T\to\bP^1_F$, Lemma~\ref{lemm-dp-base-change-b} gives
	\[
	\rho_F(X)\leq b(F,Y,f^*(-K_X))
	=\rho_{F(T)}(X_{F(T)})\leq10-d.
	\]
	Finally, $\rho_F(X)\geq2$ implies $d\leq8$.
\end{proof}

\begin{rema}\label{rema-dp-candidates}
	Suppose that $X$ is a smooth Fano threefold, and let $\eta$
	be the generic point of $B$.
	For each $\tau_N\in P_\pi$, Lemma~\ref{lemm-dp-base-change-b}
	gives the following criterion:
	\[
	\tau_N\text{ is breaking}
	\quad\Longleftrightarrow\quad
	\rho_{F(N)}(X_{F(N)})\geq\rho_F(X).
	\]
	By construction, the contraction parametrized by $N$ becomes
	defined after base change but is not defined over $F(B)$.
	It therefore supplies a new Galois-invariant geometric divisor
	class. In the conic-bundle case, this class may have a Brauer
	obstruction to descent, but a positive multiple always descends.
	Thus, after tensoring with $\bQ$, it gives a new divisor class over
	$F(N)$, and hence
	\[
	\rho_{F(N)}(X_{F(N)})>\rho_{F(B)}(X_\eta).
	\]
	Consequently, every member of $P_\pi$ is breaking whenever
	\[
	\rho_F(X)=\rho_{F(B)}(X_\eta)+1.
	\]
\end{rema}
\begin{coro}\label{cor-dp-rank-two}
	Let $X$ be a smooth Fano threefold over $F$ with $\rho_F(X)=2$
	and a del Pezzo fibration $\pi:X\to B$.
	Let $f:Y\to X$ be a dominant $a$-cover whose general geometric
	adjoint Iitaka fiber maps onto a general fiber of $\pi$.
	Then $f$ is breaking if and only if it factors rationally
	over $X$ through a member of $P_\pi$.
\end{coro}

\begin{proof}
	The forward implication is Theorem~\ref{theo-breaking}.
	Conversely, Proposition~\ref{prop-breaking-surfaces-move}
	identifies $Y$ birationally over $X$ with $X\times_B T$.
	A factorization through $\tau_N\in P_\pi$ induces
	$T\dashrightarrow N$, so Lemma~\ref{lemm-dp-base-change-b}
	and Remark~\ref{rema-dp-candidates} give
	\[
	b(F,Y,f^*(-K_X))
	=\rho_{F(T)}(X_{F(T)})
	\geq \rho_{F(N)}(X_{F(N)})
	\geq 2.
	\]
	Since the adjoint Iitaka dimension is one, $f$ is breaking.
\end{proof}

\section{Del Pezzo fibrations with dense geometric exceptional sets}
\label{sec-dp-families}

\subsection{Generic fibers and blow-up types}
\label{sec:bt-models}

Throughout this section, let $X$ be a smooth complex Fano threefold,
general in its deformation family,
and let $\pi:X\to B=\mathbb P^1$ be a del Pezzo fibration.
Write $K=\mathbb C(B)$, $S=X_K$, and $d=K_S^2$.
We assume that $S$ is nonsplit.

We define several notations which we call blow-up types.
Write $M_d$ ($1\le d\le5$) for
a del Pezzo surface of degree $d$ with $\rho_K=1$ and
maximal Galois action on its Picard group.
Write
$M_d(a_1,\ldots,a_s)$ for a blow-up of $M_d$ at closed points of degrees
$a_1,\ldots,a_s$ in general position.
We assume that the $i$-th closed point has the maximal Galois action of
$S_{a_i}$. We write $Q_+$ for $\mathbb P^1\times\mathbb P^1$ and write
$Q_-$ for a nonsplit quadric.
We write $C_4$ for a minimal quartic del Pezzo surface with $\rho_K=2$
and Galois action $W(D_4)$.

For a type $\mathbb P^2(1^k,m)$ or $Q_+(m)$ with $m>1$, let
$c$ be the genus of the normalization of $B$ in the residue field
of its degree-$m$ point. For $Q_-$ and $Q_-(2)$, let $q$ be the
genus of the ruling double cover.

\begin{lemm}\label{lem:bt-nodal-fibers}
	Let $X$ be a smooth complex Fano threefold, general in its
	deformation family, and let $\pi:X\to\mathbb P^1$ be a del Pezzo
	fibration listed in Table~\ref{tab:bt-models}.
	Then every singular fiber of $\pi$ is integral and has exactly
	one singular point, which is of type $A_1$.
\end{lemm}
\begin{proof}
	By scanning Table~\ref{tab:bt-models}, one sees that
	all rows satisfy \(\rho_K(S)=\rho(X)-1.\)
	Thus every fibration listed in Table~\ref{tab:bt-models}
	has integral fibers.

	We now prove that every singular fiber has exactly one $A_1$ singularity.
	When $\pi$ has relative Picard rank $1$, the assertion follows from
		\cite[proof of Corollary~5.14]{burke2022geometric}.
		
		{\bf($Q_-,Q_-(2),Q_+(2)$)}
		For the quadric fibrations on Nos.~2--18, 2--25 and 2--29,
		smoothness of $X$ forces every singular quadric
		to have rank three.
		The fibrations $(3\text{--}11,7)$, $(4\text{--}4,6)$ and
		$(5\text{--}1,5)$ are obtained from these by blowing up
		disjoint sections.
		A section lies in the smooth locus of $\pi$.
		Thus the assertion follows.

		{\bf(Conic-bundles)}
		For $(3\text{--}1,4)$ and $(3\text{--}2,3),(3\text{--}2,6)$,
		the discriminants $\Delta$ of the conic bundles on them
		are smooth, and the projections of $\Delta$
		to $B$ are simply branched, by
		\cite[proof of Theorem~5.13 and proof of Lemma~11.7]
		{burke2022geometric}, respectively.
		For $(3\text{--}4,4)$, the discriminant is a smooth plane
		quartic, and its projection is the pencil of lines through
		a general point. Hence it is simply branched by
		\cite[Exp.~XVII, Th.~2.5.2]{SGA7II}.
		In each case, singular fibers of
		$\pi$ have a unique $A_1$ singularity.
		
		{\bf(Blow-ups at a degree $2$ multi-section)}
		Let $Y$ be a blow-up
		of $\pi:X\to B$ at a degree $2$ center.
		When $\pi$ is smooth, the types of $Y$ are
		\[
		(4\text{--}1,6),\ (4\text{--}2,6),\
		(4\text{--}3,6),\ (4\text{--}5,6),\
		(3\text{--}15,7),\ (3\text{--}18,7),\
		(3\text{--}21,7).
		\]
		In each case, singular fibers of
		$Y\to B$ have a unique $A_1$ singularity.
		When $\pi$ is not smooth, these are
		$(3\text{--}6,6)$, $(3\text{--}10,6)$,
		and $(3\text{--}4,6)$.
		In these cases, a general blow-up center $C$
		avoids the non-smooth points of $\pi$,
		and its branch locus avoids the singular values
		of $\pi$. The assertion follows.
		
		{\bf(Blow-ups at a degree$\geq3$ multi-section)}
		It remains to treat the seven presentations with horizontal
		centers of degree at least three:
		\[
		\begin{gathered}
			(3\text{--}3,5),\ (3\text{--}5,4),\
			(3\text{--}6,5),\ (3\text{--}7,6),\\
			(3\text{--}8,5),\ (3\text{--}12,6),\
			(4\text{--}13,5).
		\end{gathered}
		\]
		Their underlying surface bundles are smooth, so it suffices
		to prove that the induced map \(C\to B\) is simply branched.
		For $(3\text{--}6,5)$, write
		$X=\operatorname{Bl}_{C\sqcup\ell}\mathbb P^3$,
		where $C$ is a smooth quartic elliptic curve and $\ell$ is a
		general disjoint line.
		The fibration is induced by the planes through $\ell$.
		The restriction map
		\[
		H^0(\mathbb P^3,\mathcal O_{\mathbb P^3}(1))
		\xrightarrow{\sim} H^0(C,\mathcal O_C(1))
		\]
		is an isomorphism,
		and $\mathcal O_C(1)$ is very ample.
		Projective duality therefore identifies a general line $\ell$
		with a general pencil in the complete very ample linear system
		$|\mathcal O_C(1)|$.
		By \cite[Exp.~XVII, Th.~2.5.2]{SGA7II}, a general
		hyperplane pencil for this embedding is a Lefschetz pencil,
		hence the induced map  $C\to  B$ is simply branched.
		The same theorem applies to five of the remaining cases,
		using the following very ample line bundles:
		\[
		\begin{aligned}
			(3\text{--}5,4)&:\quad
			(C,L)=(\mathbb P^1,\mathcal O_{\mathbb P^1}(5)),\\
			(3\text{--}8,5)&:\quad
			(C,L)=(\mathbb P^1,\mathcal O_{\mathbb P^1}(4)),\\
			(3\text{--}12,6)&:\quad
			(C,L)=(\mathbb P^1,\mathcal O_{\mathbb P^1}(3)),\\
			(4\text{--}13,5)&:\quad
			(C,L)=(\mathbb P^1,\mathcal O_{\mathbb P^1}(3)),\\
			(3\text{--}7,6)&:\quad
			g(C)=1,\ \deg L=3.
		\end{aligned}
		\]
		
		Finally, No.~3--3 is a general hypersurface
		\[
		X=\{s_0Q_0(t,x)+s_1Q_1(t,x)=0\}
		\subset\mathbb P^1_s\times\mathbb P^1_t\times\mathbb P^2,
		\qquad
		Q_i\in H^0(\mathbb P^1_t\times\mathbb P^2,\mathcal O(1,2)).
		\]
		Projection to $\mathbb P^1_s$ resolves a general pencil
		in the very ample system $|\mathcal O(1,2)|$.
		The Lefschetz pencil theorem therefore gives the assertion.
\end{proof}

\begin{prop}\label{prop:bt-models}
	Let $X$ be a smooth complex Fano threefold, general in its deformation
	family.
	Table~\ref{tab:bt-models} lists all del Pezzo fibrations on $X$ with nonsplit
	generic fiber, together with the blow-up type of the generic fiber.
	The multiplicity is the number of distinct extremal rays of
	$\operatorname{Nef}(X)$ defining del Pezzo fibrations of the
	indicated type.
	
	Assume that $X$ is general in its deformation family.
	Writing $N$ for the number of singular fibers, we have
	\begin{equation}\label{eq:bt-N}
		N=22-2d-2\rho(X)+2h^{1,2}(X).
	\end{equation}
	For the following blow-up types, we have
	\begingroup\normalfont
	\begin{equation}\label{eq:N}
		N=\begin{cases}
			2c+2m-2, & S\text{ of type }\mathbb P^2(1^k,m)\text{ or }Q_+(m),\\
			2q+2, & S\text{ of type }Q_-,\\
			2q+4, & S\text{ of type }Q_-(2).
		\end{cases}
	\end{equation}
	\endgroup
\end{prop}

\begin{longtable}{@{}l c c c c@{}}
	\caption{Del Pezzo fibrations with nonsplit generic fiber on a general
		member of each deformation family of smooth complex Fano threefolds.}
	\label{tab:bt-models}\\
	\toprule
	No. & $d$ & blow-up type & parameter & multiplicity\\
	\midrule\endfirsthead
	\toprule No. & $d$ & blow-up type & parameter & multiplicity\\
	\midrule\endhead
	\bottomrule\endfoot
	2--1 & $1$ & $M_1$ & $N=60$ & $1$\\
	2--2 & $2$ & $M_2$ & $N=54$ & $1$\\
	2--3 & $2$ & $M_2$ & $N=36$ & $1$\\
	2--4 & $3$ & $M_3$ & $N=32$ & $1$\\
	2--5 & $3$ & $M_3$ & $N=24$ & $1$\\
	3--2 & $3$ & $M_4(1)$ & $N=16$ & $1$\\
	2--7 & $4$ & $M_4$ & $N=20$ & $1$\\
	2--10 & $4$ & $M_4$ & $N=16$ & $1$\\
	3--1 & $4$ & $C_4$ & $N=24$ & $3$\\
	3--4 & $4$ & $C_4$ & $N=12$ & $1$\\
	3--5 & $4$ & $\mathbb P^2(5)$ & $c=0$ & $1$\\
	2--14 & $5$ & $M_5$ & $N=10$ & $1$\\
	3--3 & $5$ & $\mathbb P^2(4)$ & $c=3$ & $2$\\
	3--6 & $5$ & $\mathbb P^2(4)$ & $c=1$ & $1$\\
	3--8 & $5$ & $\mathbb P^2(4)$ & $c=0$ & $1$\\
	4--13 & $5$ & $\mathbb P^2(1,3)$ & $c=0$ & $1$\\
	5--1 & $5$ & $\mathbb P^2(1,1,2)$ & $c=0$ & $1$\\
	3--2 & $6$ & $Q_-(2)$ & $q=3$ & $1$\\
	3--4 & $6$ & $Q_-(2)$ & $q=2$ & $1$\\
	3--6 & $6$ & $Q_-(2)$ & $q=1$ & $1$\\
	3--10 & $6$ & $Q_-(2)$ & $q=0$ & $2$\\
	3--7 & $6$ & $\mathbb P^2(3)$ & $c=1$ & $1$\\
	3--12 & $6$ & $\mathbb P^2(3)$ & $c=0$ & $1$\\
	4--1 & $6$ & $Q_+(2)$ & $c=1$ & $4$\\
	4--2 & $6$ & $\mathbb P^2(1,2)$ & $c=1$ & $2$\\
	4--3 & $6$ & $Q_+(2)$ & $c=0$ & $1$\\
	4--4, 4--5 & $6$ & $\mathbb P^2(1,2)$ & $c=0$ & $1$\\
	3--11 & $7$ & $\mathbb P^2(2)$ & $c=1$ & $1$\\
	3--15, 3--18, 3--21 & $7$ & $\mathbb P^2(2)$ & $c=0$ & $1$\\
	2--18 & $8$ & $Q_-$ & $q=2$ & $1$\\
	2--25 & $8$ & $Q_-$ & $q=1$ & $1$\\
	2--29 & $8$ & $Q_-$ & $q=0$ & $1$\\
\end{longtable}

\begin{proof}
	Put $A=-K_X$.
	A del Pezzo fibration on $X$ is given by a nonzero nef class $F$
	with $AF^2=0$. Such nef class must generate an extremal ray of
	$\operatorname{Nef}^1(X)$.
	Such classes $F$ can be read from Matsuki's tables \cite[\S III-3]{Mat95},
	with the corrections in \cite[\S\S3,4,6]{Mat23}.
	For each $F$, we choose a sequence of blow-downs over $B$.
	Restricting this sequence to the generic fiber gives
	the blow-up types in Table~\ref{tab:bt-models}.

	Let $\Lambda=\operatorname{Pic}(S_{\overline K})$ and let
	$R$ be the image of $\operatorname{Pic}(X)_{\mathbb Q}$ in
	$\Lambda_{\mathbb Q}$. By \cite[Theorem~5.13]{burke2022geometric},
	\begin{equation}\label{eq:bt-G}
		G=\operatorname{Fix}_{W(S)}(R).
	\end{equation}
	Assume that $X$ is general in its deformation family.
	By Lemma~\ref{lem:bt-nodal-fibers}, a smooth fiber has Euler
	characteristic $12-d$, whereas each singular fiber has Euler
	characteristic $11-d$. Hence
	\[
	2+2\rho(X)-2h^{1,2}(X)
	=e(X)
	=(2-N)(12-d)+N(11-d)
	=24-2d-N,
	\]
	which proves \eqref{eq:bt-N}.
	
	For the types $\mathbb P^2(1^k,m)$ and $Q_+(m)$,
	let $C\to B$ be the normalization of $B$
	in the residue field of the degree-$m$ closed point, and put
	$c=g(C)$. Since blowing up a degree-$m$ point decreases the
	anticanonical degree by $m$,
	\[
	m=
	\begin{cases}
		9-d-k, & S\text{ of type }\mathbb P^2(1^k,m),\\
		8-d,   & S\text{ of type }Q_+(m).
	\end{cases}
	\]
	By Lemma~\ref{lem:bt-nodal-fibers}, $C\to B$
	is simply branched at exactly
	the $N$ singular values. Riemann--Hurwitz gives
	\[
	2c-2=-2m+N,
	\]
	proving $N=2c+2m-2$.

	For $Q_-(2)$, the monodromy group is generated by the involutions
	that exchange the rulings and the blown-up points.
	The two corresponding double covers have genera
	$q$ and $0$, respectively. Their branch loci are disjoint
	and are exactly the values of singular fibers.
	In each row with blow-up type $Q_-(2)$, the
	blow-up center corresponding to the degree $2$ point
	is a smooth rational curve. Hence the corresponding
	double cover has two branch points by Riemann--Hurwitz.
	Thus we have
	\[
	N=(2q+2)+2=2q+4.
	\]
	A simplified argument shows that $N=2q+2$ for $Q_-$.

	Substituting the values of $\rho(X)$ and $h^{1,2}(X)$ in
	\eqref{eq:bt-N} gives $N$ for every row of Table~\ref{tab:bt-models}.
	The remaining parameters and the multiplicities
	are given by the geometric descriptions for in each row.
\end{proof}

\subsection{Parameter covers of the blow-up types}
\label{sec:bt-covers}

Let $\Omega$ be either the set $\mathscr R$ of fiber classes of
conic bundles on $S_{\overline K}$, or, for a fixed $r$, the set
$\mathscr E^{(r)}$ of unordered collections of $r$ pairwise
disjoint $(-1)$-curves on $S_{\overline K}$.
Let $T\to B$ be the connected Galois cover defined by the full
monodromy group $G$.

The group $G$ acts on $\Omega$, and each $G$-orbit
$G/H\subset\Omega$ gives a connected component
$D=T/H$ of $P_\pi$.
Write $n=[G:H]$ and let
\[
f_j=\#\{gH\in G/H\mid s_jgH=gH\}
\]
be the number of cosets fixed by the local reflection $s_j$.
Let $N_j$ be the number of
branch values whose local monodromy is conjugate to $s_j$.
Then Riemann--Hurwitz gives
\begin{equation}\label{eq:bt-genus}
	\deg(D/B)=n,\qquad
	g(D)=1-n+\frac14\sum_jN_j(n-f_j).
\end{equation}
Indeed, above such a value there are $f_j$ unramified points
and $(n-f_j)/2$ points of ramification index two.
Suppose that $G=\prod_j S_{m_j}$ and that
$H=\prod_j\bigl(\prod_a S_{\lambda_{ja}}\bigr)$ is a product
of Young subgroups, with $\sum_a\lambda_{ja}=m_j$.
For $s_j$ a transposition in the $j$th factor, we have
\begin{equation}\label{eq:bt-young}
	n=\prod_j\frac{m_j!}{\prod_a\lambda_{ja}!},\qquad
	\frac{f_j}{n}=
	\frac{\sum_a\binom{\lambda_{ja}}2}{\binom{m_j}2}.
\end{equation}

For a cover $D\to B$, write $\mu(D)=(c_{\rm cb};c_1,\ldots,c_{9-d})$,
where the entries count its occurrences in $\mathscr R$ and
$\mathscr E^{(r)}$. The symbols in the column $D$
of Table~\ref{tab:bt-covers} have the following meanings.
\begin{itemize}
	\item $C$: for $\mathbb P^2(1^k,m)$ or $Q_+(m)$, the
	normalization of $B$ in the residue field of the degree-$m$
	closed point. It marks one of its $m$ geometric points.
	We thus have $\deg(C/B)=m$ and $g(C)=c$.
	\item $C_{[2]}$: for $\mathbb P^2(4)$, the cover marking an
	unordered pair of distinct points in a geometric fiber of
	$C\to B$.
	We have $\deg(C_{[2]}/B)=\binom42=6$.
	\item $Q$: for $Q_-$ or $Q_-(2)$, the double cover marking
	one of the two ruling classes of the underlying quadric
	$Q_-$. Thus $g(Q)=q$.
	\item $E$: for $Q_-(2)$, the double cover marks one
	of the two exceptional curves. We have $\deg(E/B)=2$ and $g(E)=0$.
	\item $T$: the connected Galois cover of $B$ defined
	by the full monodromy representation, with Galois group $G$.
	\item $C_5$: for $M_5$, the degree-five cover marking one
	of the five conic-bundle classes on $S_{\overline K}$.
	Its stabilizer is $S_4\subset S_5$.
	Since a transposition fixes three of the five conic bundles,
	and $N=10$, we have
	\[
	g(C_5)=1-5+\frac{10}{4}(5-3)=1.
	\]
\end{itemize}
For covers with too high genus, we left a $-$ in the column of
cover.

\begin{prop}\label{prop:bt-covers}
	For types $M_1,M_2,M_3,M_4,M_4(1),C_4$, every
	non-trivial component of $P_{\pi}$ has genus at least two.
	Table~\ref{tab:bt-covers} lists every
	non-trivial component of $P_{\pi}$
	for all other types.
\end{prop}

\begin{longtable}{@{}l c r c c@{}}
	\caption{Parameter components of degree greater than one.
		The parameters $c,q$ are supplied by Table~\ref{tab:bt-models}.}
	\label{tab:bt-covers}\\
	\toprule
	blow-up type & $D$ & $n$ & $g(D)$ & $\mu$\\
	\midrule\endfirsthead
	\toprule blow-up type & $D$ & $n$ & $g(D)$ & $\mu$\\
	\midrule\endhead
	\bottomrule\endfoot
	$M_5$ & $C_5$ & 5 & $1$ & $(1;0,0,0,1)$\\
	$$ & $-$ & 10 & $6$ & $(0;1,0,1,0)$\\
	$$ & $-$ & 20 & $16$ & $(0;0,0,1,0)$\\
	$$ & $-$ & 30 & $31$ & $(0;0,1,0,0)$\\
	\midrule
	$\mathbb P^2(5)$ & $C$ & 5 & $0$ & $(2;1,0,0,2,1)$\\
	$$ & $-$ & 10 & $3$ & $(0;1,2,2,2,1)$\\
	$$ & $-$ & 20 & $9$ & $(0;0,0,1,3,0)$\\
	$$ & $-$ & 30 & $19$ & $(0;0,2,2,1,0)$\\
	$$ & $-$ & 60 & $49$ & $(0;0,0,1,0,0)$\\
	\midrule
	$\mathbb P^2(4)$ & $C$ & 4 & $c$ & $(1;1,0,3,1)$\\
	$$ & $C_{[2]}$ & 6 & $2c+1$ & $(0;1,1,1,0)$\\
	$$ & $-$ & 12 & $5c+4$ & $(0;0,2,1,0)$\\
	\midrule
	$\mathbb P^2(1,3)$ & $C$ & 3 & $0$ & $(1;3,6,7,1)$\\
	$$ & $T$ & 6 & $1$ & $(0;0,2,1,0)$\\
	\midrule
	$\mathbb P^2(1,1,2)$ & $C$ & 2 & $0$ & $(1;3,12,10,1)$\\
	\midrule
	$\mathbb P^2(3)$ & $C$ & 3 & $c$ & $(1;2,3,0)$\\
	\midrule
	$\mathbb P^2(1,2)$ & $C$ & 2 & $c$ & $(1;2,3,0)$\\
	\midrule
	$Q_+(2)$ & $C$ & 2 & $c$ & $(0;3,3,1)$\\
	\midrule
	$Q_-(2)$ & $E$ & 2 & $0$ & $(0;1,1,1)$\\
	$$ & $Q$ & 2 & $q$ & $(1;0,1,0)$\\
	$$ & $T$ & 4 & $2q+1$ & $(0;1,1,0)$\\
	\midrule
	$\mathbb P^2(2)$ & $C$ & 2 & $c$ & $(1;1,0)$\\
	\midrule
	$Q_-$ & $Q$ & 2 & $q$ & $(1;0)$\\
\end{longtable}

\begin{proof}
	For each of the six types of low degrees, the monodromy representation
	is the reflection representation of
	$E_8,E_7,E_6,D_5,D_5,D_4$, respectively, and we have $N\ge12$.
	Thus for any \(v\in\Omega\) with a nontrivial \(G\)-orbit, the stabilizer
	$H=\operatorname{Stab}_G(v)$ is contained in a maximal proper parabolic
	subgroup $P$ of $G$. For these root systems, the proportion
	of fixed cosets in $G/P$ is $|\Phi_P|/|\Phi|\le3/5$.
	Consequently, writing $n_P=[G:P]$,
	the Riemann-Hurwitz formula
	\[
	2g(T/P)-2=-2n_P+N\frac{n_P-f_P}{2}.
	\]
	gives
	\[
	g(D)\ge g(T/P)\ge1-n_P+\frac{12n_P}{4}(1-3/5)
	=1+\frac{n_P}{5}>1,
	\]
	which excludes the low degree types.

	For the blow-up types listed in Table~\ref{tab:bt-covers}, the monodromy
	group $G$ is a product of symmetric groups, and
	each factor permutes a set given by the defining data.
	For example, for $Q_-(2)$, we have
	$G=S_2\times S_2$, and the two factors exchange the two ruling classes and
	the two exceptional curves, respectively.
	Each $v\in\Omega$ gives a partition of the set,
	and $H$ consists of all permutations
	preserving each part of the partition.
	Hence $H=\operatorname{Stab}_G(v)$ is a product of Young subgroups,
	and \eqref{eq:bt-young} gives their genera.
	For example, for $\mathbb P^2(4)$, the connected parameter
	covers of degree greater than one, up to isomorphism
	over $B$, mark the following data:
	\[
	\{e_1\},\qquad
	\{h-e_1-e_2\},\qquad
	\{h-e_1-e_2,\ h-e_1-e_3\},
	\]
	and the corresponding covers have degrees $4,6,12$, respectively.
	Since $N=2c+6$ by \eqref{eq:N}, formula~\eqref{eq:bt-genus} gives their genera
	as $c$, $2c+1$, and $5c+4$, respectively.
	
	For $M_5$, the connected parameter covers
	of degree greater than one mark the following data:
	\[
	\{2h-e_1-e_2-e_3-e_4\},\qquad
	\{e_1\},\qquad
	\{e_1,e_2,e_3\},\qquad
	\{e_1,e_2\},
	\]
	where the first one is a conic bundle and
	the remaining markings are collections of disjoint lines.
	The corresponding covers have degrees $5,10,20,30$,
	respectively, and a transposition in the monodromy group
	$S_5$ fixes $3,4,6,6$ markings in the respective orbits.
	Since $N=10$, formula~\eqref{eq:bt-genus} gives their genera
	as $1,6,16,31$, respectively.
	
	For $Q_-(2)$, let $q_1,q_2$ be the ruling classes and
	$e_1,e_2$ the exceptional classes.
	The covers $E,Q,T$ are represented, respectively, by
	\[
	\{e_1\},\qquad \{q_1\},\qquad \{q_1-e_1\}.
	\]
	They mark an exceptional curve, a ruling class, and the
	strict transform of a ruling line through a blown-up point,
	respectively, and have degrees $2,2,4$.
	There are $2q+2$ branch values with local monodromy $s_q$
	and two with local monodromy $s_e$, where $s_q$ exchanges
	the ruling classes and $s_e$ exchanges the exceptional classes.
	These branch loci are disjoint.
	Thus \eqref{eq:bt-genus} reads
	\[
	g(D)=1-n+\frac14\bigl((2q+2)(n-f_q)+2(n-f_e)\bigr).
	\]
	Substituting $(n,f_q,f_e)=(2,2,0),(2,0,2),(4,0,0)$
	gives $g(E)=0$, $g(Q)=q$, and $g(T)=2q+1$, respectively.
	
	Write $H=\operatorname{Stab}_G(v)$.
	The occurrences of the cover $T/H$ in the parameter scheme
	$P_\pi$ correspond precisely to the
	$G$-orbits in $\Omega$ whose point stabilizers are conjugate
	to $H$.
	Each of these orbits has $[G:H]$ elements.
	Hence the corresponding entry of $\mu(T/H)$ is
	\[
	c_\Omega(v)=
	\frac{\#\{w\in\Omega\mid
		\operatorname{Stab}_G(w)\text{ is conjugate to }H\}}
	{[G:H]},
	\qquad
	\Omega=\mathscr R\ \text{or}\ \mathscr E^{(r)}.
	\]
	These calculations give every row and complete the proof.
\end{proof}

The following theorem gives the precise family lists and
the arithmetic conclusion summarized in
Theorem~\ref{intro-dp-families}.

\begin{theo}[Family lists and potential density]\label{thm:bt-families}
	Let $X$ be a smooth complex Fano threefold, general in its
	deformation family. For $i\in\{0,1\}$, there exist a del Pezzo fibration $\pi$ and
	a connected parameter cover $D\to\mathbb P^1$ of degree greater
	than one with $g(D)=i$ if and only if the family number belongs
	to $\mathcal I_i$, where
	\begingroup\normalfont
	\[
	\begin{aligned}
		\mathcal I_0=\{&2\text{--}29,\ 3\text{--}2,\ 3\text{--}4,\ 3\text{--}5,
		\ 3\text{--}6,\ 3\text{--}8,\ 3\text{--}10,\ 3\text{--}12,\\
		&3\text{--}15,\ 3\text{--}18,\ 3\text{--}21,\ 4\text{--}3,
		\ 4\text{--}4,\ 4\text{--}5,\ 4\text{--}13,\ 5\text{--}1\},\\[3pt]
		\mathcal I_1=\{&2\text{--}14,\ 2\text{--}25,\ 3\text{--}6,\ 3\text{--}7,
		\ 3\text{--}8,\ 3\text{--}10,\\
		&3\text{--}11,\ 4\text{--}1,\ 4\text{--}2,\ 4\text{--}13\}.
	\end{aligned}
	\]
	\endgroup
	
	For each of the $22$ families in $\mathcal I_0\cup\mathcal I_1$,
	the points represented by models over number fields with potentially
	Zariski-dense geometric exceptional sets are Zariski dense
	in the corresponding geometric deformation space.
\end{theo}

\begin{proof}
	The two lists follow by substituting the parameters of
	Table~\ref{tab:bt-models} into Table~\ref{tab:bt-covers} and applying
	Proposition~\ref{prop:bt-covers}. Their intersection is
	$\{3\text{--}6,3\text{--}8,3\text{--}10,4\text{--}13\}$,
	so their union has $16+10-4=22$ elements. Choose $U$ so that
	Proposition~\ref{prop:bt-models} and Proposition~\ref{prop:bt-covers}
	hold throughout $U$.
	
	Let $X/F$ represent a point of $U$. After a finite extension of
	$F$, choose $\pi:X\to B=\mathbb P^1_F$ and a geometrically
	integral parameter cover $D\to B$ with $g(D)\le1$, both defined
	over $F$.
%	Let $f:Y\to X$ be a smooth
%	projective model of $X\times_B D$.
%	Since the involved del Pezzo fibrations has degree at least $t$, the generic fiber of $Y\to D$ is geometrically rational.
%	So after a finite extension of $F$, the threefold $Y$ is rational.
%	After a further finite extension of $F$, we may assume that
%	$X$ is split.

	Let $f:Y\to X$ be a smooth projective resolution of
	$X\times_B D$, and let $q:Y\to D$ be the induced fibration.
	Its geometric generic fiber is rationally connected, so the
	Graber--Harris--Starr theorem gives a section of $q_{\overline F}$.
	After a finite extension of $F$, we may assume that this section
	is defined over $F$, that $D(F)$ is infinite, and that $X$ is split.
	The generic fiber of $q$ is birational to a smooth del Pezzo
	surface of degree at least $4$ with an $F(D)$-rational point.
	It is therefore $F(D)$-unirational, yielding a dominant rational map
	\[
	\mathbb P^2_F\times D\dashrightarrow Y
	\]
	over $D$.
	
	Since the fibers of $\pi$ are
	integral,
	Proposition~\ref{prop:bt-covers} implies that
	\[
	\rho_{F(D)}(X_{F(D)})\ge \rho_{F(B)}(S)+1=\rho_F(X).
	\]
	By Lemma~\ref{lemm-dp-base-change-b},
	\[
	a(Y,f^*(-K_X))=1,
	\qquad
	b(F,Y,f^*(-K_X))
	=\rho_{F(D)}(X_{F(D)})\geq\rho_F(X).
	\]
	Since $\deg(D/B)>1$, the same lemma identifies $Y\to D$
	with the adjoint Iitaka fibration, up to birational equivalence.
	Thus $f$ is a breaking thin map and $f(Y(F))\subset Z(X)$.
	
	Finally, since $D(F)$ is infinite, the set
	$(\mathbb P^2_F\times D)(F)$ is Zariski dense.
	The dominant rational map above therefore implies that
	$Y(F)$ is Zariski dense.
	Consequently, $f(Y(F))\subset Z(X)$ is Zariski dense in $X$.
	The last assertion follows because $U(\overline{\mathbb Q})$
	is Zariski dense and each of its points has a model over
	a number field.
\end{proof}

\section{Zariski dense geometric exceptional sets arising from del Pezzo fibrations}
\label{sec-example}
\subsection{Quadric surface fibrations}

For a quadric-surface fibration on a smooth Fano threefold,
a rank-three fiber is a Du Val del Pezzo surface of type $A_1$.
Each geometric irreducible component of a fiber of rank at most
two has higher $a$-invariant by Proposition~\ref{prop-dp-bad-fibers}.
If such a component is defined over $F$, the rational-point image
of its resolution is contained in $Z_a(X)$.

For a smooth quadric surface, the only parameter cover in
Construction~\ref{cons-dP-fib} is the degree-two scheme of rulings.
We use this scheme to define the discriminant cover, rather than
specifying only its branch divisor, which would leave a quadratic
twist undetermined.

\begin{coro}\label{coro:p1xp1}
	Let $X$ be a smooth Fano variety over $F$ and let $\pi:X\to B$
	be a flat quadric-surface fibration over a positive-dimensional smooth
	projective variety.
	Let $q:D\to B$ be the normal discriminant cover
	parametrizing rulings, and suppose that $D$ is geometrically integral.
	Let $\tau:Y_D\to X$ be a smooth projective model of the projection
	$X\times_B D\to X$, and write $p:Y_D\to D$ for its resolved map.
	Assume that $p$ is birationally the adjoint Iitaka fibration of
	$(Y_D,\tau^*(-K_X))$. Then
	\[
	a(Y_D,\tau^*(-K_X))=1,
	\qquad b(F,Y_D,\tau^*(-K_X))=2.
	\]
	Thus $\tau$ is breaking if and only if $\rho_F(X)=2$.
	Every dominant breaking map whose canonical fibers map onto quadric
	fibers factors rationally through $\tau$.
	
	If $X$ is a smooth Fano threefold, the assumption is
	automatic. In this case, if $\rho_F(X)=2$, a smooth fiber $X_t$ is
	breaking exactly when $t\in B(F)$ lifts to $D(F)$.
\end{coro}

\begin{proof}
	A geometric quadric surface has no $(-1)$-curves and has exactly two
	conic-bundle classes, namely its rulings. Consequently $P_\pi$ consists
	of $\tau$. Over $F(D)$ both ruling classes are invariant, so the
	generic Picard rank is two. Lemma~\ref{lemm-dp-base-change-b} gives the
	claimed invariants. The pair has positive adjoint Iitaka dimension,
	and $\rho_F(X)\geq2$, proving the breaking criterion. Factorization
	follows from Theorem~\ref{theo-breaking}.
	For a Fano threefold apply the last assertion of
	Lemma~\ref{lemm-dp-base-change-b}. A smooth quadric fiber has arithmetic
	Picard rank two precisely when its two ruling classes are individually
	invariant, or equivalently when its ruling scheme has an $F$-point.
	This proves the last assertion.
\end{proof}

We recall two examples from the literature with Zariski dense
exceptional sets arising from quadric-surface fibrations.
\begin{exam}\label{exam-BT-quadric}
	The quadric analogue of the Batyrev--Tschinkel example is
	the Fano fivefold
	\[
	\sum_{i=0}^3x_iy_i^2=0
	\quad\subset\bP^3\times\bP^3,
	\]
	whose Manin asymptotic was established in \cite{BTB20}.
	Restricting the first projection to general lines gives
	Fano threefolds of type No.~2-25, and this construction
	dominates their geometric deformation family.
	Their Manin--Peyre asymptotic was proved under additional
	arithmetic hypotheses in \cite[Theorem~1.1]{BBH24}.
\end{exam}
\begin{exam}
	Let $X=\Bl_C(Q)$, where
	\[
	Q=V(z_{11}z_{12}-z_{21}^2-z_{31}z_{32})
	\subset\bP^4_{\bQ},
	\qquad
	C=Q\cap V(z_{31},z_{32}).
	\]
	This is a Fano threefold of type No.~2-29, whose Manin
	asymptotic was established in
	\cite[Theorem~1.1]{blomer2024manin}.
	The pencil $[z_{31}:z_{32}]=[s:t]$ has quadric fibers
	\[
	z_{11}z_{12}-z_{21}^2-stz^2=0.
	\]
	Its ruling cover $w^2=-st$ is rational, and
	$[z_{11}:z_{12}:z_{21}:z]=[1:0:0:0]$ defines a section.
	Hence infinitely many rational parameters give smooth
	split fibers.
	Since $\rho_{\bQ}(X)=2$, these fibers are breaking by
	Corollary~\ref{coro:p1xp1}, and their rational points
	are Zariski dense in $X$.
\end{exam}

\subsection{A quartic del Pezzo fibration}\label{sec-dp4}

In this section, we give a general construction of Fano fourfolds with Zariski dense
exceptional sets, arising from quartic del Pezzo fibrations.

Let $V$ be a four-dimensional vector space over a number field $F$.
Let
\[
\mathcal N=\langle q_0,q_1,q_2\rangle
\]
be a net of quadrics in $\mathbb P(V)\cong\mathbb P^3$ whose base locus contains a rational point $z_0\in\bP(V)(F)$. Put
\[
B=\mathbb P(\mathcal N)\cong\mathbb P^2.
\]
For $\boldsymbol{t}=[t_0:t_1:t_2]\in B$, write
\[
q_{\boldsymbol t}=t_0q_0+t_1q_1+t_2q_2.
\]
After choosing a basis of $V$, let
\[
A(\boldsymbol t)=t_0A_0+t_1A_1+t_2A_2
\]
be the symmetric matrix representing $q_{\boldsymbol t}$.
We assume from now on that
\[ 
\det A(\boldsymbol t)\in F(B)^\times\setminus F(B)^{\times2}.
\]
Then the discriminant curve of
the net is defined as
\[
\Delta_{\mathcal N}
=
\{\det A(\boldsymbol{t})=0\}
\subset B.
\]
Let
\(
\delta\colon D_{\mathcal N}\longrightarrow B
\)
be the normalization of the discriminant double cover
\(
w^2=\det A(\boldsymbol{t}).
\)

Set $V^+=V\oplus F$ and regard $q_t$ as a quadratic form on $V^+$
independent of the last coordinate. Thus
\[
\widehat Q_{\boldsymbol{t}}=\{q_{\boldsymbol{t}}=0\}
\subset\bP(V^+)\cong\bP^4
\]
is a quadric cone with vertex $v=[0:0:0:0:1]$.

Let $P\in\bP(V^+)$ be a lifting of $z_0$.
Let
\[
Q^3=V(Q)\subset\mathbb P(V^+)
\]
be a smooth quadric threefold containing $P$.
We define a variety
\[
X=
\left\{
(\boldsymbol{t},x)\in B\times\bP(V^+)
\ \middle|\
x\in Q,\quad q_{\boldsymbol{t}}(x)=0
\right\}.
\]

\begin{defi}\label{def-admissible}
	The datum $(\mathcal N,Q)$ is admissible if it satisfies the following conditions.
	\begin{enumerate}
		\item The base locus of $\mathcal{N}$ is a reduced scheme $Z$ of dimension $0$.\label{def-admissible2}
		\item $Q(v)\neq0$, and for each $z\in Z$, the quadric $Q$ intersects $L_z$ transversely, where $L_z$ is the line connecting $z$ and $v$. \label{def-admissible3}
	\end{enumerate}
\end{defi}

\begin{lemm}
	After fixing $z_0$ and $P$,
	the admissible locus is a nonempty Zariski open subset of the parameter space of $(\mathcal{N},Q)$.
\end{lemm}
\begin{proof}
	After choosing bases, the net $\cN$ can be put into the form
	\[
	q_i=x_0x_i+r_i(x_1,x_2,x_3).
	\]A general triple of quadrics through \(z_0\) has reduced zero-dimensional base locus of length \(2^3=8\). This proves that condition (\ref{def-admissible2}) holds on a nonempty open subset.
	It is clear that condition (\ref{def-admissible3}) defines a nonempty open subset of the parameter space.
\end{proof}

\begin{lemm}\label{lemm-X is smooth}
	Assume that the datum $(\cN,Q)$ satisfies conditions (\ref{def-admissible2}) and (\ref{def-admissible3}) of Definition \ref{def-admissible}. Then \(X\) is smooth.
\end{lemm}

\begin{proof}
	Consider the smooth fivefold
	\[
	M=Q^3\times B.
	\]The variety \(X\subset M\) is the zero locus of the section
	\[
	s(x,\boldsymbol t)
	=
	t_1q_1(x)+t_2q_2(x)+t_3q_3(x)
	\]of \(\mathcal O_M(2,1)\).
	Suppose first that, at a point \((x,\boldsymbol t)\in X\), not all of the \(q_i(x)\) vanish. Then the derivative of \(s\) in the \(B\)-directions is
	\[
	d_Bs=(q_1(x),q_2(x),q_3(x)),
	\]which is nonzero. Hence \(X\) is smooth at \((x,\boldsymbol t)\).
	It remains to consider points for which
	\[
	q_1(x)=q_2(x)=q_3(x)=0.
	\]Since the \(q_i\) are independent of \(x_4\), their common zero locus in \(\bP^4\) is the cone over \(Z\) with vertex \(v\):
	\[
	V(q_1,q_2,q_3)
	=
	\bigcup_{z\in Z}L_z.
	\]Because \(Q(v)\neq0\), the point \(x\) is not equal to \(v\). Since \(Z\) is a reduced scheme of dimension $0$, the common kernel of
	\[
	dq_1(x),\ dq_2(x),\ dq_3(x)
	\]is precisely the tangent direction \(T_xL_z\).
	By assumption, \(Q^3\) meets \(L_z\) transversely at \(x\). Hence
	\[
	T_xL_z\not\subset T_xQ^3.
	\]It follows that the restrictions
	\[
	dq_1(x)|_{T_xQ^3},\quad
	dq_2(x)|_{T_xQ^3},\quad
	dq_3(x)|_{T_xQ^3}
	\]are linearly independent. Therefore, for every nonzero \(\boldsymbol t\),
	\[
	d_xs
	=
	t_1dq_1+t_2dq_2+t_3dq_3
	\]does not vanish on \(T_xQ^3\). Thus \(X\) is smooth also at such points.
\end{proof}

Under the assumption of the above lemma, the variety $X$ is smooth and the first projection defines a quartic del Pezzo fibration
\[
\pi\colon X\longrightarrow B.
\]
Let $H$ and $L$ denote the restrictions of the hyperplane classes from
$B$ and $\mathbb P(V^+)$, respectively. By adjunction,
\[
-K_X=2H+L.
\]
Thus $X$ is a Fano fourfold. Moreover, viewing $X$ as a divisor of class
$H+2L$ on $B\times Q$, the Lefschetz hyperplane theorem gives
\[
\operatorname{Pic}(X)=\mathbb ZH\oplus\mathbb ZL.
\]

Let $\Delta_\pi\subset B$ denote the discriminant locus of
$\pi:X\to B$, and put
\[
B^\circ=B\setminus(\Delta_{\mathcal N}\cup\Delta_\pi).
\]
For $\boldsymbol t\in B^\circ$, the quadric
$\widehat Q_{\boldsymbol t}$ has rank four and
\[
X_{\boldsymbol t}=Q\cap\widehat Q_{\boldsymbol t}
\]
is a smooth quartic del Pezzo surface.
The two geometric rulings of $\widehat Q_{\boldsymbol t}$ by planes
induce two conic-bundle classes on $X_{\boldsymbol t}$.
Thus the restriction of $D_{\mathcal N}\to B$ over $B^\circ$
parametrizes these two classes.
Since \(z_0\) belongs to the base locus of \(\mathcal N\), every cone
\(\widehat Q_{\boldsymbol t}\) contains the line \(L_{z_0}\). Hence the
point
\[
P\in Q^3(F)\cap L_{z_0}(F)
\]
defines a constant section of
\[
\pi\colon X\longrightarrow B.
\]

Recall that the discriminant cover
\[ \delta:D_{\cN}\to B \]
is defined by the normalization of
\[ w^2=\det A(x). \]
If the discriminant curve \(\Delta_{\mathcal N}\) is smooth, then
\(D_{\mathcal N}\) is a smooth del Pezzo surface of degree \(2\).
In general \(D_{\mathcal N}\) may be singular, as in Example~\ref{prop-split-diagonal-dp4} below.

\begin{prop}\label{prop-dp4-rational}
	Assume that condition~(\ref{def-admissible2}) of
	Definition~\ref{def-admissible} holds. Then \(D_{\mathcal N}\) is
	\(F\)-rational.
\end{prop}
\begin{proof}
	Let $Z$ be the base locus of $\cN$, and let $\Pi\simeq\mathbb P^2_F$
	parametrize lines through $z_0$. We first determine where the
	restriction map
	$\Phi:\cN\longrightarrow H^0(\ell,\mathcal O_\ell(2)(-z_0))$
	drops rank. Work over $\bar F$.
	Let \(v\) be a linear form on \(\ell\) vanishing at \(z_0\). Write \(q_i|_\ell=vL_i\), with \(L_i\in H^0(\ell,\mathcal O_\ell(1))\). Since at least one \(q_i|_\ell\) vanishes to order exactly one at \(z_0\), the \(L_i\) have no common zero at \(z_0\).
	So $\Phi$ has rank one
	if and only if $L_i$ have a common zero. Such a zero is
	precisely a second point of $Z$ on $\ell$. Thus the restriction map
	has rank two away from the finitely many lines joining $z_0$ to
	another geometric base point.
	
	Consider the incidence variety
	\[
	I=\{(\ell,t)\in\Pi\times B\mid \ell\subset q_t\}.
	\]
	Over the rank-two locus \(U\subset\Pi\) of $\Phi$, the incidence variety is the
	graph of the morphism \(U\to B\) defined by \(\ell\mapsto\ker\Phi_\ell\).
	Let \(B'\subset B\) be the complement of \(\Delta_{\cN}\) and the finitely
	many pencils \(\mathbb P(\ker\Phi_\ell)\) for \(\ell\notin U\).
	For each geometric point \(t\in B'\), each of the two rulings of \(q_t\)
	contains a unique line through \(z_0\). Sending each such line to its
	ruling identifies \(I|_{B'}\) with \(D_{\cN}|_{B'}\).
	In other words, the relative Fano scheme of lines over \(B'\) has Stein factor
	\(D_{\cN}|_{B'}\).
	By construction, \(I|_{B'}\) is a nonempty open subset of the graph,
	and the first projection identifies the graph with \(U\) over \(F\).
	Hence \(D_{\cN}\) is \(F\)-birational to \(\Pi\simeq\mathbb P^2_F\).
\end{proof}

\begin{prop}\label{prop-net-dp4-density}
	Assume that $X$ is smooth.
	Let $Y$ be a smooth projective model of the main component of
	$X\times_BD_{\cN},$
	and let
	\[
	\tau\colon Y\longrightarrow X
	\]
	be the induced map.
	Suppose the following assertions hold.
	\begin{enumerate}
		\item The fibration $\rho:Y\ra D_{\cN}$ admits an $F$-rational section.
		\item $D_{\mathcal N}$ has Zariski dense rational points.
	\end{enumerate}
	Then $\tau$ is a breaking thin map and
	$\tau(Y(F))$ is Zariski dense in $X$. In particular, the geometric
	exceptional set of $X$ is Zariski dense.
\end{prop}

\begin{proof}
	The map $\tau$ is the base change associated with the degree-$2$
	component $D_{\mathcal N}\to B$ of $\mathscr R\to B$. Hence it is a
	member of $P_\pi$.

	By birational invariance, we may choose $Y$ with generic fiber
	$S=X_{F(D_{\mathcal N})}$ over $D_{\mathcal N}$.
	Lemma~\ref{lemm-dp-base-change-b} gives
	$a(Y,\tau^*(-K_X))=1$.
	
	Over $F(D_{\mathcal N})$, the two rulings of the quadric cone
	induce two conic-bundle structures on $S$.
	Spreading out these families gives nef curve classes
	$\alpha_1,\alpha_2\in N_1(Y)$.
	Since $\alpha_1$ and $\alpha_2$ are linearly independent,
	\[
	b(F,Y,\tau^*(-K_X))\geq2=b(F,X,-K_X).
	\]
	Moreover, $\tau_*\alpha_1=\tau_*\alpha_2$, so $\tau$ is
	face contracting.
	Definition~\ref{defi-ges} shows that $\tau$ is breaking.
	
	By assumption, the fibration
	\[
	\rho\colon Y\longrightarrow D_{\mathcal N}
	\]
	admits an \(F\)-rational section. Hence its generic fiber has an
	\(F(D_{\mathcal N})\)-rational point.
	A quartic del Pezzo surface over a field
	of characteristic zero with a rational point is unirational
	\cite[Theorem~30.1]{manin1986cubicforms}. After spreading out the
	resulting parametrization, we obtain a dominant rational map
	\[
	\mathbb P^2_{D_{\mathcal N}}\dashrightarrow Y.
	\]
	Since $D_{\mathcal N}(F)$ is Zariski dense, it follows that $Y(F)$,
	and hence $\tau(Y(F))$, is Zariski dense.
\end{proof}

We summarize the preceding discussion in the following theorem.
\begin{theo}\label{theo-dp4}
	Let
	\[
	\pi\colon X\longrightarrow\mathbb P^2_F
	\]
	be a quartic del Pezzo fibration obtained by intersecting a general
	net of quadric cones with common vertex \(v\) and a general smooth
	quadric threefold
	\[
	Q^3\subset\mathbb P^4_F.
	\]
	Assume that the base locus of the corresponding net in
	\(\mathbb P^3\) contains a point
	\(
	z_0\in\mathbb P^3(F),
	\)
	and that
	\[
	Q^3(F)\cap L_{z_0}(F)\neq\varnothing.
	\]
	Then
	\[
	\left\{
	t\in\mathbb P^2(F)
	\ \middle|\
	X_t\text{ is smooth and }\rho(X_t)\geq2
	\right\}
	\]
	is Zariski dense in \(\mathbb P^2\).
\end{theo}
\begin{proof}
	The point of \(Q^3(F)\cap L_{z_0}(F)\) defines a section of \(\pi\).
	Moreover, the discriminant double cover is \(F\)-rational by
	Proposition~\ref{prop-dp4-rational} and hence has Zariski dense rational points.
	The assertion follows from
	Proposition~\ref{prop-net-dp4-density}.
\end{proof}

We now give explicit examples satisfying the hypotheses of
Proposition~\ref{prop-net-dp4-density} and hence having
Zariski dense exceptional sets.
The density hypothesis on $D_{\mathcal N}$ is particularly easy to
verify for split diagonal nets.

\begin{exam}[Split diagonal nets]\label{prop-split-diagonal-dp4}
	Let
	\[
	L_i=\{\ell_i=0\}\subset \mathbb P^2_F,
	\qquad 0\leq i\leq 3,
	\]
	be four lines defined over $F$ and in general position.
	The generality assumption implies that their defining forms satisfy a
	unique linear relation in which every coefficient is nonzero.
	Thus, after rescaling the $\ell_i$, we may assume that
	\[
	\ell_0+\ell_1+\ell_2+\ell_3=0.
	\]
	Consider the diagonal net
	\[
	q_x(y)
	=
	\ell_0(x)y_0^2+\ell_1(x)y_1^2
	+\ell_2(x)y_2^2+\ell_3(x)y_3^2.
	\]
	Then
	\[
	q_x(\pm1,\pm1,\pm1,\pm1)=0
	\]
	for every $x\in\mathbb P^2$. Thus condition
	(\ref{def-admissible2}) of Definition \ref{def-admissible} is satisfied.
	
	Write
	\[
	L_Z=
	\left\{
	[a_0:a_1:a_2:a_3:a_4]\in\mathbb P^4
	\;\middle|\;
	a_0^2=a_1^2=a_2^2=a_3^2
	\right\}.
	\]
	Let $Q\subset\mathbb P^4$ be a smooth quadric threefold such that
	$Q$ intersects $L_Z$ transversely and $Q\cap L_Z$ contains an
	$F$-rational point. Then condition (\ref{def-admissible3}) of
	Definition \ref{def-admissible} is satisfied.
	
	Define
	\[
	X=
	\left\{
	(x,y)\in\mathbb P^2\times\mathbb P^4
	\;\middle|\;
	Q(y)=q_x(y)=0
	\right\}.
	\]
	Then $X$ is a smooth Fano fourfold of Picard rank $2$, and the first
	projection defines a quartic del Pezzo fibration
	\[
	\pi\colon X\longrightarrow\mathbb P^2.
	\]
	
	The discriminant cover of the net is
	\[
	D=
	\left\{
	w^2=\ell_0\ell_1\ell_2\ell_3
	\right\}
	\longrightarrow\mathbb P^2,
	\]
	whose branch curve is the union of four lines in general position.
	Thus $D$ is a del Pezzo surface of degree $2$ with six singularities
	of type $A_1$.
	
	On the open subset where $\ell_1\ell_3\neq0$, put
	\[
	u=\frac{\ell_0}{\ell_1},
	\qquad
	v=\frac{\ell_2}{\ell_3},
	\qquad
	z=\frac{w}{\ell_1\ell_3}.
	\]
%	The relation
%	\[
%	\ell_0+\ell_1+\ell_2+\ell_3=0
%	\]
%	shows that \(u\) and \(v\) are birational coordinates on
%	\(\mathbb P^2\).
	Since \(z^2=uv\), we have \(F(D)=F(u,v)(z)=F(u,z)\).
	Hence \(D\) is \(F\)-rational and \(D(F)\) is Zariski
	dense.
\end{exam}

\begin{prop}\label{prop-split-diagonal-dp4-exceptional}
	Let $\pi\colon X\to\mathbb P^2$ be the quartic del Pezzo fibration
	constructed in Example \ref{prop-split-diagonal-dp4}, and let
	$B^\circ\subset\mathbb P^2$ denote its smooth locus.
	Then
	\[
	\bigcup_{\substack{x\in B^\circ(F)\\ \rho(X_x)\geq2}}
	X_x(F)
	\]
	is Zariski dense in \(X\). In particular, the geometric exceptional
	set of \(X\) is Zariski dense.
\end{prop}

\begin{proof}
	The discriminant cover $D$ is $F$-rational by
	Example \ref{prop-split-diagonal-dp4}, and hence $D(F)$ is Zariski
	dense. The assertions therefore follow immediately from
	Proposition \ref{prop-net-dp4-density}.
\end{proof}

\begin{exam}\label{exam-diagonal-dp4}
	Take homogeneous coordinates $[x_0:x_1:x_2]$ on $\mathbb P^2$ and
	$[y_0:\cdots:y_4]$ on $\mathbb P^4$. Put
	\[
	q_x(y)
	=
	x_0(y_0^2-y_3^2)
	+x_1(y_1^2-y_3^2)
	+x_2(y_2^2-y_3^2)
	\]
	and
	\[
	Q(y)=
	y_0^2+y_1^2+y_2^2+y_3^2-y_4^2.
	\]
	Define
	\[
	X=
	\{Q(y)=q_x(y)=0\}
	\subset\mathbb P^2\times\mathbb P^4.
	\]
	
	We first check that this is a member of the family in
	Example \ref{prop-split-diagonal-dp4}.  In this case
	\[
	\ell_0=x_0,\qquad
	\ell_1=x_1,\qquad
	\ell_2=x_2,\qquad
	\ell_3=-x_0-x_1-x_2.
	\]
	The base locus $L_Z$ is the union of the lines
	\[l_{\boldsymbol{\epsilon}}\colon
	\bP^1\to\bP^4,\qquad
	[a:b]\mapsto [a:\epsilon_1a:\epsilon_2a:\epsilon_3a:b],
	\qquad
	\epsilon_i\in\{\pm1\}.
	\]
	On each such line, the restriction of $Q$ is
	\[
	4a^2-b^2.
	\]
	Hence $Q\cap L_Z$ consists of $F$-rational transverse intersection
	points. Therefore the hypotheses of
	Example \ref{prop-split-diagonal-dp4} are satisfied. In particular,
	$X$ is a smooth Fano fourfold of Picard rank $2$.
	
	It follows from
	Proposition~\ref{prop-split-diagonal-dp4-exceptional} that the
	geometric exceptional set of \(X\) is Zariski dense.
	
	More explicitly, let $B^\circ\subset\bP^2$ be the locus
	of smooth fibers. The following set is a Zariski dense
	thin subset of $Z(X)$:
	\[
	T=
	\left\{
	(x,y)\in X(F)\mid
	x\in B^\circ(F),\
	-x_0x_1x_2(x_0+x_1+x_2)\in F^{\times2}
	\right\}.
	\]
\end{exam}

\begin{coro}[Potential density for threefold sections]
	\label{rema-dp-threefold-sections}
	Let $\pi:X\to\bP^2_F$ be as above,
	with general defining data, and let $\ell\subset\bP^2_F$
	be a general line.
	Then $X_\ell:=\pi^{-1}(\ell)$ is a smooth Fano threefold
	of type No.~2-7, and its geometric exceptional set is
	potentially Zariski dense.
\end{coro}

\begin{proof}
	The threefold $X_\ell$ is the blow-up of $Q^3$ along the
	smooth base curve of a pencil in $|\mathcal O_{Q^3}(2)|$,
	hence is of type No.~2-7 and has geometric Picard rank $2$.
	Its Picard group is generated by restrictions of divisor
	classes defined over $F$, so this rank does not increase
	after a field extension.
	
	The discriminant double cover restricts to a smooth
	genus-one curve
	\[
	C_\ell\longrightarrow\ell.
	\]
	After a finite extension $K/F$, the curve $C_\ell$ has a
	$K$-point and its Jacobian has a $K$-rational point of
	infinite order. Thus $C_\ell(K)$ is infinite.
	By the preceding construction, its image contains infinitely
	many parameters of smooth fibers $S$ satisfying
	\[
	\rho_K(S)\geq2=\rho_K(X_\ell),
	\]
	and each such fiber has Zariski dense $K$-rational points.
	Since $-K_{X_\ell}|_S=-K_S$, we have
	\[
	a(S,-K_{X_\ell}|_S)=1,
	\qquad
	b(K,S,-K_{X_\ell}|_S)=\rho_K(S).
	\]
	Hence these fibers are breaking.
	Their rational points form a Zariski dense subset of
	$X_{\ell,K}$, proving the assertion.
\end{proof}

\begin{rema}[Comparison with the Batyrev--Tschinkel example]
	The present construction is similar in spirit to the example of
	Batyrev and Tschinkel \cite{BT96}. In their setting, the diagonal
	presentation of the cubic-surface fibers provides an effective way to
	describe the Galois action on the $27$ lines; see
	\cite{peyre2001tamagawa}. 
	However, for non-diagonal cubic surfaces, controlling the monodromy action may be difficult even in special families; see, for example, \cite{medrano2022monodromy}.
\end{rema}

\subsection{A degree-six del Pezzo fibration}
In this subsection, we construct an explicit Fano threefold
of type No.~3-10 over any number field $F$.
It admits two degree-six del Pezzo fibrations, each with
infinitely many split fibers of Picard rank four.
These fibers give a Zariski dense geometric exceptional set over $F$
and a lower bound of order $B(\log B)^3$ on every non-empty
open subset.

\begin{exam}[Fano threefold No~3-10]
	\label{exam-symmetric-3-10}
	Let $F$ be a number field, and take homogeneous coordinates
	$[x_0:x_1:y_0:y_1:z]$ on $\bP^4_F$. Put
	\[
	Q=\{x_0x_1+y_0y_1=z^2\},\qquad
	C_x=Q\cap\{x_0=x_1=0\},\qquad
	C_y=Q\cap\{y_0=y_1=0\}.
	\]
	The quadric $Q$ is smooth, and $C_x$ and $C_y$ are disjoint smooth
	conics. Let
	\[
	\mu:X=\operatorname{Bl}_{C_x\sqcup C_y}Q\longrightarrow Q,
	\]
	and denote the exceptional divisors by $E_x$ and $E_y$.
	This is a member of the Mori--Mukai family No.~3-10.
	
	The two projections $[x_0:x_1]$ and $[y_0:y_1]$ extend to morphisms
	$\pi_x,\pi_y:X\to\bP^1$.
	Write $M=\mu^*\mathcal O_Q(1)$,
	$A=\pi_x^*\mathcal O_{\bP^1}(1)=M-E_x$, and
	$B=\pi_y^*\mathcal O_{\bP^1}(1)=M-E_y$. Then
	\[
	L:=-K_X=3M-E_x-E_y=M+A+B,
	\qquad \rho_F(X)=\rho_{\overline F}(X)=3.
	\]
	Thus $L$ is very ample and $X$
	is a smooth Fano threefold with
	$(a(X,L),b(F,X,L))=(1,3)$.
	
	For $t\in F^\times$, the fiber $S_t=\pi_y^{-1}([t:1])$ is the
	blow-up of
	\[
	Q_t=\{x_0x_1+t y^2=z^2\}
	\subset\bP^3_{[x_0:x_1:y:z]}
	\]
	at the degree-two subscheme $\{x_0=x_1=0\}$. Its two geometric
	points do not lie on a common ruling, so $S_t$ is a smooth del Pezzo
	surface of degree six. If $t=s^2$ with $s\in F^\times$, then
	\[
	x_0x_1=(z-sy)(z+sy),\qquad
	p_\pm=[0:0:1:\pm s].
	\]
	Hence $Q_t$ is split and both blow-up points are $F$-rational.
	If $t\notin F^{\times2}$, the nontrivial element of
	$\operatorname{Gal}(F(\sqrt t)/F)$ exchanges both the two rulings
	and the two exceptional curves. Consequently,
	\begin{equation}\label{eq-symmetric-3-10-ranks}
		\rho_F(S_t)=
		\begin{cases}
			4,&t\in F^{\times2},\\
			2,&t\notin F^{\times2}.
		\end{cases}
	\end{equation}
	Since $L|_{S_t}=-K_{S_t}$, the square parameters give
	\[
	\bigl(a(S_t,L|_{S_t}),b(F,S_t,L|_{S_t})\bigr)
	=(1,4)>(1,3).
	\]
	The parameter cover is the rational double cover
	\[
	\delta:\bP^1\longrightarrow\bP^1,
	\qquad [a:b]\longmapsto[a^2:b^2].
	\]
	It is also the map $C_x\to\bP^1$ induced by $\pi_y$, under the
	parametrization $[a:b]\mapsto[0:0:a^2:b^2:ab]$.
	The same statements hold for $\pi_x$ by symmetry.
\end{exam}

\begin{prop}\label{prop-symmetric-3-10-exceptional}
	Let $X$ be as in Example~\ref{exam-symmetric-3-10}, and put
	$F^2=\{s^2 \mid s\in F\}$, including zero. Then
	\begin{equation}\label{eq-symmetric-3-10-exceptional}
		Z(X)=
		\left\{
		P\in X(F)\;\middle|\;
		\begin{array}{l}
			\mu(P)=[x_0:x_1:y_0:y_1:z],\\
			x_0x_1\in F^2\ \text{or}\ y_0y_1\in F^2
		\end{array}
		\right\}.
	\end{equation}
	In particular, $Z(X)$ is Zariski dense in $X$ over $F$.
\end{prop}

\begin{proof}
	Put
	\[
	\Delta=\mu^{-1}\bigl(Q\cap\{x_0x_1y_0y_1=0\}\bigr),
	\qquad X^\circ=X\setminus\Delta.
	\]
	We first describe the contributions supported on $\Delta$.
	Every $F$-point of $E_x\cup E_y$ lies on an $F$-rational exceptional
	fiber, which has anticanonical degree one.
	Let $q\in Q(F)\setminus(C_x\cup C_y)$ satisfy
	$x_0x_1y_0y_1=0$. We may first assume that $x_0=0$.
	The section $Q\cap\{x_0=0\}$ is a rank-three quadric cone
	with vertex on $C_y$ and base $C_x$.
	The ruling line through $q$ is defined over $F$ and meets
	each blow-up center transversely in one point.
	Its strict transform therefore has anticanonical degree
	one. These show that $\Delta(F)\subset Z_a(X)$.
	
	Next, we show that $\Delta(F)=Z_a(X)$.
	The exceptional fibers of $\mu$ and the ruling
	lines above give curve classes dual to $A,B,M$, respectively.
	Thus
	\[
	\operatorname{Nef}(X)
	=\mathbb R_{\geq0}M+\mathbb R_{\geq0}A+\mathbb R_{\geq0}B.
	\]
	The corresponding extremal contractions are the two blow-downs
	and the conic bundle
	$(\pi_x,\pi_y):X\to\bP^1\times\bP^1$.
	In particular, $X$ has no E2 divisor.

	Since $L=M+A+B$, every anticanonical line is contracted
	either by $\mu$ or by $(\pi_x,\pi_y):X\to \bP^1\times\bP^1$.
	In the latter case, it is a component of a degenerated fiber.
	Hence every anticanonical line lies in $\Delta$.
	Corollary~\ref{cor-a-exceptional-set}, together with the
	inclusion proved above, gives $Z_a(X)=\Delta(F)$.

	Since $\rho_F(X)=3$, the conic contribution in
	Proposition~\ref{prop-breaking-map-reduction}(1) is absent.
	The nef cone computed above shows that $\pi_x$ and $\pi_y$
	are the only del Pezzo fibrations, up to isomorphism of
	the base.
	All their singular fibers are contained in $\Delta$.
	
	Smooth del Pezzo surfaces admit no adjoint rigid $a$-covers.
	Thus a smooth fiber contributes precisely when its Picard
	rank is at least three, which occurs exactly over a square
	parameter by \eqref{eq-symmetric-3-10-ranks}.
	
	It remains to consider dominant breaking base changes.
	For either fibration, write $t$ for the affine parameter.
	The same calculation of the Galois action shows that, over
	any finite extension $K/F(t)$, the generic fiber has Picard
	rank four if $t\in K^{\times2}$ and rank two otherwise.
	By Lemma~\ref{lemm-dp-base-change-b}, a breaking base change
	must therefore factor through the double cover
	given by $t=s^2$.
	Its rational-point image outside $\Delta$ consequently
	lies over square parameters.
	
	Proposition~\ref{prop-breaking-map-reduction}, together
	with $Z_a(X)=\Delta(F)$, now proves
	\eqref{eq-symmetric-3-10-exceptional}.
	The smooth fibers over nonzero square parameters are split
	rational surfaces and occur over infinitely many parameters,
	so their rational points are Zariski dense in $X$.
\end{proof}

\begin{coro}\label{coro-symmetric-3-10-manin}
	Let $X$ be as in Example~\ref{exam-symmetric-3-10}. For every
	non-empty open subset $U\subset X$ and every anticanonical height
	function $H$ on $X$, we have
	\[
	N(U(F),H,B)\gg_{U,H}B(\log B)^{3}.
	\]
	In particular,
	\[
	\lim_{B\to\infty}
	\frac{N(U(F),H,B)}{B(\log B)^{\rho_F(X)-1}}
	=+\infty.
	\]
	Thus the original form of Manin's conjecture, in which only a
	proper closed subset is removed, fails on every non-empty open
	subset of $X$, for every anticanonical height.
\end{coro}

\begin{proof}
	Choose $s\in F^\times$ such that $S=S_{s^2}$ meets $U$.
	The surface $S$ is a split toric del Pezzo surface with
	$\rho_F(S)=4$, and $L|_S=-K_S$. Let $T\subset S$ be its open
	torus. Manin's conjecture for toric varieties
	\cite{batyrev1998manin}, together with bounded comparability of
	anticanonical heights, gives
	\[
	N((U\cap T)(F),H|_S,B)\gg_{U,H}B(\log B)^3.
	\]
	All anticanonical lines on $S$ lie in the toric boundary.
	Thus every curve meeting $T$ has anticanonical degree
	at least two, and the proper closed subset $T\setminus U$
	contributes $O_{U,S,H}(B)$ points.
	Consequently,
	\[
	N(U(F),H,B)
	\geq N((U\cap T)(F),H|_S,B)
	\gg_{U,H}B(\log B)^3.
	\]
	Since $\rho_F(X)=3$, the second assertion follows.
\end{proof}

\section{Anticanonical conics on Fano threefolds of geometric Picard rank $1$}\label{sec-curves}

In this section, we study adjoint rigid curves with the same
$a$-invariant as the ambient Fano threefold, with particular attention
to the Zariski density of the geometric exceptional set
(Definition~\ref{defi-ges}).

\subsection{Conics on Fano threefolds of index $1$ and Picard rank $1$}

In this section we study anticanonical conics on Fano threefolds of Picard rank $1$,
which are conics on threefolds of index $1$ or
lines on threefolds of index $2$.

Let $X$ be a Fano threefold of index $1$ and Picard rank $1$. There are
ten deformation families of such threefolds, indexed by the genus
\[
g(X)\in\{2,3,4,5,6,7,8,9,10,12\}.
\]

We first record a general consequence of the Mordell--Lang theorem.

\begin{lemm}\label{lemm-irregular-surface}
	Let $S$ be a smooth projective geometrically integral surface over a
	number field $F$. If $q(S)=h^1(S,\mathcal O_S)>2$, then $S(F)$ is not
	Zariski dense.
\end{lemm}

\begin{proof}
	Suppose that $S(F)$ is Zariski dense, and choose a point of $S(F)$ as
	the base point of the Albanese morphism
	$\alpha:S\to\operatorname{Alb}(S)$. By the Mordell--Weil and
	Mordell--Lang theorems, $\alpha(S)$ is a translate of an abelian
	subvariety. Since it contains the origin and generates
	$\operatorname{Alb}(S)$, the morphism $\alpha$ is surjective. This
	contradicts $q(S)=\dim\operatorname{Alb}(S)>\dim S$.
\end{proof}

Since $\rho_{\mathrm{geom}}(X)=1$, there is no del Pezzo
fibration on $X_{\overline F}$.
By Remark~\ref{rema-ges-density-reduction}, it suffices to
study the conic contribution.

\subsubsection{Geometry of the parameter spaces of conics}

We first recall the description of the Fano surface of conics in genus
at least $7$.

\begin{prop}[{\cite[Theorem~1.1.1]{kuznetsov2018hilbert}}]
	\label{prop-S(X)}
	Let $X$ be a Fano threefold of index $1$, Picard rank $1$, and genus
	$g\geq7$ over $\bC$. Then the Fano surface $S(X)$ of conics is smooth
	and irreducible. Moreover,
	\begin{enumerate}
		\item if $g=12$, then
		\[
		S(X)\cong\bP^2;
		\]
		
		\item if $g=10$, then $S(X)$ is an abelian surface;
		
		\item if $g=9$, then $S(X)$ is a ruled surface over a smooth curve
		of genus $3$;
		
		\item if $g=8$, then $S(X)$ is a minimal surface of general type
		with
		\[
		q(S(X))=5,\qquad p_g(S(X))=10,\qquad K_{S(X)}^2=45;
		\]
		
		\item if $g=7$, then
		\[
		S(X)\cong\operatorname{Sym}^2(C),
		\]
		where $C$ is a smooth curve of genus $7$. In particular,
		\[
		q(S(X))=7.
		\]
	\end{enumerate}
\end{prop}

For the remaining genera, we use the following results on the dominant
components of the space of conics.

\begin{prop}\label{prop-low-genus-conics}
	Let $X$ be a Fano threefold of index $1$ and Picard rank $1$ over
	$\bC$.
	
	\begin{enumerate}
		\item Suppose that $g=3,4$, or $5$, that $-K_X$ is very ample, and
		that $X$ is general in moduli. Then the space of conics has a unique
		dominant component, which is birational to a smooth projective
		surface $S(X)$. Its irregularity is
		\[
		\begin{array}{c|ccc}
			g(X) & 3 & 4 & 5\\ \hline
			q(S(X)) & 30 & 20 & 14.
		\end{array}
		\]
		
		\item Suppose that $g=6$ and that $X$ is general in moduli. Then
		the Fano surface of conics is a smooth irreducible surface of
		general type satisfying
		\[
		q(S(X))=10.
		\]
		
		\item Suppose that $g=2$ and that $X$ is general in moduli. Then
		the space of conics has two dominant components
		\[
		M_{\mathrm{tc}}
		\qquad\text{and}\qquad
		M_{\mathrm{bt}},
		\]
		corresponding respectively to totally tangent conics and bitangent
		lines of the branch sextic.
		
		A smooth projective model $\widetilde M_{\mathrm{tc}}$ satisfies
		\[
		\operatorname{Alb}(\widetilde M_{\mathrm{tc}})
		\cong J(X),
		\]
		and hence
		\[
		q(\widetilde M_{\mathrm{tc}})=\dim J(X)=52.
		\]
		The normalization $\widetilde M_{\mathrm{bt}}$ is a smooth
		irreducible surface. Its Abel--Jacobi morphism is constant.
		
	\end{enumerate}
\end{prop}

\begin{proof}
	For $g=3,4,5$, the existence and uniqueness of the dominant component
	and its relation with the Fano surface of conics follow from
	\cite[Theorem~7.1]{Lehmann2019}. The Abel--Jacobi morphism is
	generically finite onto its image by
	\cite[Lemma~7.2]{Lehmann2019}. The stated irregularities are recalled
	in \cite[Section~2.9]{corvaja2023integral}.
	
	The assertion for $g=6$ follows from
	\cite[Propositions~0.1 and~0.5]{logachev2012fano}; the excluded locus
	has codimension one in the moduli space.
	
	Suppose finally that $g=2$. Then $X$ is a double cover
	\[
	w\colon X\longrightarrow\bP^3
	\]
	branched over a smooth sextic surface $B$. The two dominant components
	are described in \cite[Theorem~6.6]{burke2022geometric}.
	
	The component $M_{\mathrm{tc}}$ generically parametrizes curves mapping
	isomorphically to conics in $\bP^3$ having even contact with $B$.
	By \cite{ceresa1986abel}, the induced homomorphism
	\[
	\operatorname{Alb}(\widetilde M_{\mathrm{tc}})
	\longrightarrow J(X)
	\]
	is an isomorphism. Since $\dim J(X)=h^{1,2}(X)=52$, we obtain
	\[
	q(\widetilde M_{\mathrm{tc}})=52.
	\]
	
	The second component is the closure in the Chow scheme of the
	curves $w^{-1}(\ell)$, where $\ell$ is a bitangent line of $B$.
	For a general $\ell$, the inverse image has two nodes and rational
	normalization, and its anticanonical degree is $2$.
	The normalization of this parameter surface is the Fano congruence
	of bitangents, which is smooth and irreducible by
	\cite[Proposition~2.4]{kim2021normalization}. Its Abel--Jacobi map
	is constant because the cycles $w^*[\ell]$ are rationally
	equivalent; see also \cite[Remark~5.1(ii)]{kim2021normalization}.
\end{proof}

Thus Lemma \ref{lemm-irregular-surface} applies to all the dominant
components occurring for general Fano threefolds of genera
$3\leq g\leq9$, as well as to the component $M_{\mathrm{tc}}$ in
genus $2$. The arithmetic of the bitangent component
$M_{\mathrm{bt}}$ is not determined by the Abel--Jacobi morphism and
remains the unresolved part of the genus-$2$ case.

\begin{ques}\label{ques-g2}
	Let $X$ be a general Fano threefold of genus $2$ over a number field
	$F$, and let $M_{\mathrm{bt}}$ be the component of the space of conics
	induced by bitangent lines of the branch sextic. Are the rational
	points on a smooth projective model of $M_{\mathrm{bt}}$ Zariski
	dense?
\end{ques}

\subsubsection{Non-density results}

We now combine the geometric descriptions and the
Mordell--Lang theorem to show the non-density results.

\begin{theo}\label{theo-index-one-conics}
	Let $X$ be a Fano threefold of index $1$ and geometric Picard rank $1$
	over a number field $F$. Assume that one of the following conditions holds:
	\begin{enumerate}
		\item $g(X)\in\{7,8,9\}$;
		\item $g(X)\in\{3,4,5,6\}$ and $X_{\overline F}$ is general in its
		deformation family.
	\end{enumerate}
	Then the geometric exceptional set of $(X,-K_X)$ is not Zariski dense.
\end{theo}

	\begin{proof}
		By the discussion preceding Proposition~\ref{prop-S(X)}, it
		suffices to consider the dominant components of the parameter
		space of conics on $X$.
		Under the stated hypotheses, Propositions~\ref{prop-S(X)}
		and~\ref{prop-low-genus-conics} show that each such component
		has a smooth projective model $S$ with $q(S)>2$.
		Lemma~\ref{lemm-irregular-surface} therefore implies that
		$S(F)$ is not Zariski dense.
		The corresponding $F$-defined conics sweep out a subset of
		dimension at most two in $X$, proving the assertion.
	\end{proof}

If $X_{\overline F}$ is general in its deformation family, Theorem
\ref{theo-index-one-conics} excludes all genera $3\leq g(X)\leq9$.
The genera not excluded are therefore $2$, $10$, and $12$; the
bitangent contribution in genus $2$ is the subject of
Question~\ref{ques-g2}.

\subsubsection{The remaining cases and dense examples}

For genera $10$ and $12$, Proposition~\ref{prop-S(X)} shows that
$S(X)_{\overline F}$ is an abelian surface when $g(X)=10$,
and is isomorphic to $\mathbb P^2_{\overline F}$ when $g(X)=12$.
In particular, the Mordell--Lang argument used in Theorem
\ref{theo-index-one-conics} does not exclude Zariski density in these
two cases. Notice, however, that potential density of the rational
points on $S(X)$ alone does not imply potential density of the
geometric exceptional set, since an $F$-rational point of $S(X)$ may
parametrize a conic without an $F$-rational point.

The Mukai--Umemura threefold provides an explicit example in which
this obstruction does not occur.

\begin{exam}\label{exam-mukai-umemura}
	Let $X_{\mathrm{MU}}$ be the model of the Mukai--Umemura threefold
	over $\bQ$ constructed below. Then the geometric exceptional set of
	$(X_{\mathrm{MU}},-K_{X_{\mathrm{MU}}})$ is Zariski dense.
\end{exam}

\begin{proof}
	Write \(X=X_{\mathrm{MU}}\).
	Let $M_{d}$ denote the space of binary forms of degree $d$. The
	Mukai--Umemura threefold admits the presentation
	\[
	X
	=
	\overline{
		\operatorname{PGL}_2\cdot(\phi_{12},1)
	}
	\subset\bP(M_{12}\oplus M_0),
	\]
	where
	\(
	\phi_{12}
	=
	xy(x^{10}+11x^5y^5-y^{10});
	\)
	see \cite[Section~8.2]{faenzi2007bundles} and
	\cite[Section~4]{kuznetsov2018prime}. In particular,
	$X$ and the action of $\operatorname{PGL}_2$ are defined
	over $\bQ$.
	
	Set
	\(
	p=(\phi_{12},1)\in X(\bQ),
	\)
	and let
	\[
	C=\overline{\mathbb G_m\cdot p}\subset X,
	\]
	where $\mathbb G_m\subset\operatorname{PGL}_2$ is the standard maximal
	torus. Then $C$ is a smooth conic defined over $\bQ$ and
	$p\in C(\bQ)$; see \cite[proof of Proposition~4.4]{kuznetsov2018prime}.
	
	For each \(g\in\mathrm{PGL}_2(\bQ)\), the conic \(gC\) contains the rational point \(gp\). As \(gC\) and \(X\) have the same \(a\)- and \(b\)-invariants, the inclusion \(gC\hookrightarrow X\) is breaking. Hence \(Z(X)\) contains the Zariski dense orbit \(\mathrm{PGL}_2(\bQ)p\).
\end{proof}

\begin{rema}
	Browning and Loughran discuss the parameter space of conics on
	Fano threefolds of genus $12$
	\cite[Section~4.2]{browning2017varieties}.
	For the Mukai--Umemura threefold, the transitive action of its
	automorphism group on smooth conics allows us to obtain a dense
	family of conics with rational points by translating a single such conic.
\end{rema}

\begin{rema}
	Double-cover constructions can yield additional dominant components
	of the space of conics.
	For example, a special Gushel--Mukai threefold
	\[
	\pi\colon X\longrightarrow V_5
	\]
	has a component of its Hilbert scheme of conics isomorphic over
	$\overline F$ to $\mathbb P^2$, parametrizing inverse images of
	lines on $V_5$
	\cite[Theorem~6.11]{jacovskis2023categorical}.
	A smooth projective model of the corresponding universal family
	is geometrically rationally connected, since its base is
	geometrically rational and its general geometric fiber is
	$\mathbb P^1$. As the evaluation map dominates $X$, the conjectural
	potential density of rational points on rationally connected
	varieties would imply that the geometric exceptional set of $X$
	is Zariski dense after a finite extension.
\end{rema}

\subsection{Lines on Fano threefolds of index $2$ and Picard rank $1$}

Throughout this subsection, let $X$ be a smooth Fano threefold over
a number field $F$ with geometric Picard rank $1$ and geometric
index $2$. Write
\[
-K_{X_{\overline F}}=2H,
\qquad
d=H^3\in\{1,2,3,4,5\}.
\]
A line on $X$ means a geometrically integral rational curve of
$H$-degree $1$.
The projective parameter space of these curves has a unique dominant
geometric component by \cite[Theorem~6.6]{burke2022geometric}; denote
it by $\Sigma(X)$. For $d\geq3$, this is the usual Hilbert scheme of
lines. In degree $1$, we work in the Chow scheme, allowing singular
rational curves.

The density of the geometric exceptional
set reduces to the density of rational points lying on these
rational curves. We say that this set is potentially Zariski dense
if it becomes Zariski dense after a finite extension of $F$.

\begin{prop}\label{prop-index-two-lines}
	The following statements hold.
	\begin{enumerate}
		\item If $d=1$ and $X_{\overline F}$ is general in moduli, then a smooth
		projective model of $\Sigma(X)$ has irregularity $42$.
		
		\item If $d=3$, then $\Sigma(X)$ is a smooth geometrically
		integral projective surface of irregularity $5$.
		
		\item If $d=4$, then $\Sigma(X)$ is a torsor under an abelian
		surface.
		
		\item If $d=5$, then
		\[
		\Sigma(X)\cong\mathbb P^2_F,
		\]
		and the class $H$ descends to $X$.
	\end{enumerate}
\end{prop}

\begin{proof}
	In degree $1$, the dominant component is birational to the Fano
	surface studied in \cite{tikhomirov1982fano}, whose irregularity is
	$42$ by Theorem~6 of that paper.
	The geometric descriptions in degrees $3$ and $4$ are recalled in
	\cite[Theorem~1.1.1]{kuznetsov2018hilbert}. In degree $4$, the
	Albanese morphism identifies the parameter surface with a torsor
	under its Albanese variety. After a finite extension, the latter
	is the Jacobian of the genus-$2$ curve associated with the pencil
	of quadrics; see \cite[Example~3.6]{browning2017varieties}.
	The assertions in degree $5$ follow from
	\cite[Lemma~2.1 and Corollary~2.2]{dubouloz2019cylindres}.
\end{proof}

The degree-$2$ case admits a stronger finiteness statement under
an explicit assumption on the branch surface.

\begin{prop}[{\cite[Theorem~B]{corvaja2023quartic}}]
	\label{prop-CZ}
	Let $B\subset\mathbb P^3_F$ be a smooth quartic surface containing
	no geometric lines. Then only finitely many bitangent lines to
	$B$ are defined over $F$.
\end{prop}

\begin{coro}\label{coro-degree-two-lines}
	Suppose that $d=2$ and that the geometric branch quartic of $X$
	contains no lines. Then only finitely many lines on $X$ are
	defined over $F$, and the geometric exceptional set of $X$ is
	not Zariski dense.
	
	In particular, these conclusions hold when $X_{\overline F}$
	is general in moduli.
\end{coro}

\begin{proof}
	After a finite extension of $F$, we may assume that $X$ admits
	a double-cover presentation
	\[
	\pi\colon X\longrightarrow\mathbb P^3
	\]
	branched along a smooth quartic surface $B$, with
	$H=\pi^\ast\mathcal O_{\mathbb P^3}(1)$.
	It suffices to prove finiteness over this extension.
	
	Let $C\subset X$ be a line. Since $H\cdot C=1$, the morphism
	$\pi|_C$ maps $C$ birationally onto a line $\ell\subset\mathbb P^3$.
	By assumption, $\ell$ is not contained in $B$.
	The inverse image of $\ell$ is geometrically reducible precisely
	when the branch divisor $B|_\ell$ is twice an effective divisor
	of degree $2$. Thus $\ell$ is a bitangent line to $B$, allowing contact order $4$
	at one point.
	
	By Proposition \ref{prop-CZ}, there are only finitely many
	possibilities for $\ell$, and each has at most two line
	components in its inverse image. Hence there are only finitely
	many lines defined over the ground field.
	Since $\rho_{\mathrm{geom}}(X)=1$, there is no del Pezzo
	fibration on $X_{\overline F}$.
	The non-density assertion therefore follows from
	Remark~\ref{rema-ges-density-reduction} and the conic
	parameter-space reduction above.
	Finally, a general smooth quartic surface contains no lines.
\end{proof}

We summarize the Zariski-density results as follows.
\begin{coro}\label{coro-potential-density}
	Let $X$ be a smooth Fano threefold over a number field $F$ with
	geometric Picard rank $1$, index $2$, and degree $d$.
	Then the following statements hold.
	\begin{enumerate}
		\item If $d=3$, then the geometric exceptional set of $X$ is not Zariski dense. The same conclusion holds for $d=1$ and $2$ provided that $X_{\overline F}$ is general in moduli.
		\item If $d=4$, then the geometric exceptional set of $X$ is
		potentially Zariski dense.
		\item If $d=5$, then the geometric exceptional set of $X$ is
		Zariski dense.
	\end{enumerate}
\end{coro}

\begin{proof}
	The case $d=2$ follows from Corollary
	\ref{coro-degree-two-lines}.
	
	Suppose that $d=1$ or $3$, with the stated generality
	assumption when $d=1$.
	By Proposition~\ref{prop-index-two-lines}, a smooth projective
	model of the unique dominant component $\Sigma(X)$ has
	irregularity greater than $2$.
	Its rational points are not Zariski dense by
	Lemma~\ref{lemm-irregular-surface}.
	Remark~\ref{rema-ges-density-reduction} proves the first
	assertion.
	
	Now suppose that $d=4$. After a finite extension $K/F$, the
	class $H$ descends and the torsor $\Sigma(X)_K$ has a rational
	point. It is then isomorphic to an abelian surface.
	Since rational points on abelian varieties are potentially
	Zariski dense, we may enlarge $K$ so that $\Sigma(X)(K)$ is
	Zariski dense.
	
	Every $K$-defined line is isomorphic to $\mathbb P^1_K$:
	it is a smooth geometrically rational curve carrying the
	degree-$1$ line bundle obtained by restricting $H$.
	Consequently, rational points are dense in the universal
	family of lines. Each line has $a$- and $b$-invariants equal to $1$,
	as does $X_K$, so its inclusion is breaking by
	Definition~\ref{defi-ges}. Since the evaluation morphism to $X_K$
	is dominant, the geometric exceptional set of $X_K$ is Zariski
	dense.
	
	Finally, if $d=5$, Proposition \ref{prop-index-two-lines} gives
	$\Sigma(X)\cong\mathbb P^2_F$ and the descent of $H$ over $F$.
	The same argument therefore applies without extending the
	ground field.
\end{proof}

\begin{rema}
	The cases $d=3$, $4$, and $5$ were also discussed in
	\cite[Section~4.2]{browning2017varieties}, which studies the density of rational
	points lying on lines and its implications for Manin's conjecture.
\end{rema}

\begin{exam}\label{exam5.3}
	Browning and Loughran construct a smooth intersection of two
	quadrics
	\[
	X\subset\mathbb P^5_{\mathbb Q}
	\]
	containing a $\mathbb Q$-line and satisfying
	\[
	\Sigma(X)\cong J(C),
	\]
	where $C$ is the associated genus-$2$ curve and
	$J(C)(\mathbb Q)$ is Zariski dense
	\cite[Example~3.6]{browning2017varieties}.
	Every $\mathbb Q$-defined line on $X$ is isomorphic to
	$\mathbb P^1_{\mathbb Q}$, and the evaluation morphism from the
	universal family of lines is dominant.
	Hence the geometric exceptional set of $X$ is Zariski dense
	over $\mathbb Q$.
\end{exam}

Thus, for a general Fano threefold of geometric Picard rank $1$
and index $2$, the geometric exceptional set is not Zariski
dense in degrees $1$, $2$, and $3$, whereas it is potentially
Zariski dense in degrees $4$ and $5$.

% Complete replacements for Sections 7 and 8.
% Uses the existing theorem environments, macros, and bibliography keys.
% Integration instructions: sections-7-8-changes.md.

\section{Anticanonical conics on non-split Fano threefolds}
\label{sec-nonsplit}

In this section, we study anticanonical conics on smooth Fano
threefolds $X$ over a number field $F$ satisfying
\begin{equation}\label{eq-nonsplit}
	\rho_F(X)=1<\rho_{\mathrm{geom}}(X).
\end{equation}
This is the non-split case in which an anticanonical conic can be
breaking; see Definition~\ref{defi-ges} and
Section~\ref{sec-moving-adjoint-rigid}. We first use the Galois action
on divisor classes to exclude five deformation families. For the
remaining families, we apply the parameter-space results of
Section~\ref{sec-curves}.

\subsection{Galois actions and numerical restrictions}

By \cite[Theorem~1.2]{prokhorov2013g}, the possible geometric
deformation families under \eqref{eq-nonsplit} are
\[
\text{Nos.~2-6, 2-12, 2-21, 2-32, 3-1, 3-13, 3-27, and 4-1}.
\]
The two families of geometric index $2$ are Nos.~2-32 and~3-27;
the others have index $1$.

\begin{lemm}\label{lemm-galois-degree}
	Suppose that integral divisor classes $H_1,\ldots,H_r$ form a basis
	of $N^1(X_{\overline F})$ and are permuted by Galois. If
	$\rho_F(X)=1$ and
	\[
	-K_{X_{\overline F}}=c(H_1+\cdots+H_r)
	\]
	for a positive integer $c$, then every curve $C\subset X$ defined
	over $F$ has anticanonical degree divisible by $cr$.
\end{lemm}

\begin{proof}
	The Galois action is transitive, since each orbit sum is invariant
	and the invariant subspace has dimension $1$. For an $F$-defined
	curve $C$, the integers $H_i\cdot C_{\overline F}$ are therefore
	all equal, say to $m$. Hence $-K_X\cdot C=crm$.
\end{proof}

For Nos.~3-1, 3-13, and~4-1, the primitive generators $H_i$ of the
nef cone are the pullbacks of the hyperplane classes from the factors
in their standard models. They form a basis of $N^1(X_{\overline F})$
and satisfy $-K_{X_{\overline F}}=\sum_i H_i$, with $r=3,3,4$,
respectively; see \cite[\S11, case~3.1]{burke2022geometric} and
\cite[Lemmas~2.3 and~2.5]{kuznetsov2024rationality}.
For No.~2-32, the geometric threefold is a $(1,1)$-divisor in
$\bP^2\times\bP^2$, with $-K_{X_{\overline F}}=2(H_1+H_2)$.
For No.~3-27, it is $(\bP^1)^3$, with
$-K_{X_{\overline F}}=2(H_1+H_2+H_3)$.
In both cases the $H_i$ are again the primitive nef generators.
Thus Galois permutes the indicated classes in all five cases.

For types Nos.~3-1, 3-13, 4-1, 2-32, and~3-27,
Lemma~\ref{lemm-galois-degree} shows that the anticanonical degree
of a curve defined over $F$ is divisible by $3$, $3$, $4$, $4$, and $6$,
respectively. Thus none of these five types contains an anticanonical
conic defined over $F$.

\subsection{Parameter spaces and non-density}

It remains to consider Nos.~2-6, 2-12, and~2-21. We recall the
geometric descriptions of their parameter spaces. Here $F_2(X)$
denotes the Hilbert scheme of conics in the anticanonical embedding;
all three cases below have very ample anticanonical class under the
stated hypotheses.

\begin{prop}\label{prop-nonsplit-parameters}
	Let $X$ satisfy \eqref{eq-nonsplit}.
	\begin{enumerate}
		\item Suppose that $X_{\overline F}$ is general of type No.~2-6.
		Then $X_{\overline F}$ is a $(2,2)$-divisor in
		$\bP^2\times\bP^2$, and
		\[
		F_2(X_{\overline F})
		\cong\bP^2\sqcup F_{(1,1)}(X_{\overline F})\sqcup\bP^2.
		\]
		The middle component is a smooth irreducible surface with
		$q=9$. The other two components parametrize fibers of the two
		projections and are exchanged by Galois actions.
		
		\item If $X_{\overline F}$ is of type No.~2-12, then
		\[
		F_2(X_{\overline F})\cong\operatorname{Sym}^2(\Gamma),
		\qquad g(\Gamma)=3,
		\]
		where $\Gamma$ is a smooth curve over $\overline F$.
		In particular, $F_2(X)$ is a smooth geometrically integral surface
		with $q=3$.
		
		\item If $X_{\overline F}$ is of type No.~2-21, then
		\[
		F_2(X_{\overline F})\cong\bP^1\times\bP^1.
		\]
	\end{enumerate}
\end{prop}

\begin{proof}
	In the first case, a conic has bidegree $(2,0)$, $(1,1)$, or
	$(0,2)$. The first and last possibilities give the two families of
	projection fibers. The description of $F_{(1,1)}$ and its irregularity
	are given in \cite[\S5 and Proposition~5.1]{iliev2017hyper}.
	The two projection classes generate the nef cone, so Galois exchanges
	them under \eqref{eq-nonsplit}. It therefore exchanges the two
	$\bP^2$ components and preserves $F_{(1,1)}$.
	The remaining assertions follow from
	\cite[Lemmas~2.4 and~2.9]{kuznetsov2024rationality}.
\end{proof}

\begin{coro}\label{coro-nonsplit-conics}
	Let $X$ satisfy \eqref{eq-nonsplit}. The union of the $F$-rational
	points on anticanonical conics defined over $F$ is not Zariski dense
	in $X$, except possibly for types Nos.~2-6 and~2-21.
	If $X_{\overline F}$ is general in its deformation family, the only
	remaining possible type is No.~2-21.
\end{coro}

\begin{proof}
	The numerical argument above excludes five of the eight types.
	For No.~2-12, Proposition~\ref{prop-nonsplit-parameters}(2) and
	Lemma~\ref{lemm-irregular-surface} imply non-density on the parameter
	surface, so the corresponding conics sweep out a subset of
	dimension at most two in $X$.
	For general No.~2-6, the two exchanged components
	have no $F$-rational points, since the three bidegree loci are disjoint.
	The same argument applies to the remaining component, which
	descends to $F$ and has irregularity $9$.
\end{proof}

\section{Proofs of the non-density results}
\label{sec-proof}

We now prove Corollaries~\ref{theo-1} and~\ref{theo-1-nonsplit}.
The conic contribution was studied in
Sections~\ref{sec-curves} and~\ref{sec-nonsplit}.
We combine those results with
Remark~\ref{rema-ges-density-reduction} and
Lemma~\ref{lemm-dp-complexity}.

\begin{proof}[Proof of Corollary~\ref{theo-1}]
	Suppose that $Z(X)$ is Zariski dense.
	By Remark~\ref{rema-ges-density-reduction}, this density
	comes from the conic contribution or from smooth breaking
	del Pezzo fibers and dominant breaking base changes.
	
	In the latter case, some del Pezzo fibration $\pi:X\to B$
	has a breaking smooth fiber or a dominant breaking base change.
	Since $X(F)\ne\varnothing$, its genus-zero base is isomorphic
	to $\bP^1_F$.
	Lemma~\ref{lemm-dp-complexity} gives
	\[
	d\leq8,\qquad \rho_F(X)\leq10-d,
	\]
	which is the first case in the statement.
	
	The conic contribution can occur only when $\rho_F(X)=1$.
	Under the split Picard hypothesis, this means
	$\rho_{\mathrm{geom}}(X)=1$. If the geometric index is at least $3$,
	there are no curves of anticanonical degree $2$. For index $1$,
	Theorem~\ref{theo-index-one-conics} leaves precisely the genera
	\[
	g\in\{2,3,4,5,6,10,12\},
	\]
	and only $g\in\{2,10,12\}$ when $X_{\overline F}$ is general.
	For index $2$, Corollary~\ref{coro-potential-density}(1) leaves
	\[
	d\in\{1,2,4,5\},
	\]
	and only $d\in\{4,5\}$ when $X_{\overline F}$ is general.
	Outside these cases, both contributions are contained in proper
	closed subsets, proving the corollary.
\end{proof}

\begin{proof}[Proof of Corollary~\ref{theo-1-nonsplit}]
	Suppose first that $X_{\overline F}$ is not a product. Since $\rho_F(X)=1$, no del Pezzo fibration is defined over $F$. Remark~\ref{rema-ges-density-reduction} reduces the assertion
	to the conic contribution.
	The parameter-space argument of Section~\ref{sec-curves},
	together with Corollary~\ref{coro-nonsplit-conics},
	gives the stated exceptions.
	
	The only geometrically product case is
	$X_{\overline F}\simeq(\bP^1)^3$.
	Here $Z_a(X)=\varnothing$ by
	Proposition~\ref{prop-a-higher}.
	The three factor classes are permuted transitively by
	Galois, since $\rho_F(X)=1$.
	The anticanonical conics are the fibers obtained by varying
	one factor, and the adjoint rigid surfaces with $a$-invariant
	one are the fibers of the three projections to $\bP^1$,
	by Proposition~\ref{prop-product-adjoint-surfaces}.
	None of their numerical classes is Galois invariant.
	
	For any remaining breaking map, the adjoint Iitaka fibration
	used in Proposition~\ref{prop-breaking-map-reduction}
	would give a geometrically integral family of such curves
	or surfaces defined over $F$.
	Its general image would have a Galois-invariant numerical
	class, a contradiction.
	This also excludes the product case.
\end{proof}

\normalem
\bibliographystyle{alpha}
\bibliography{myref25}

\end{document}